\documentclass[11pt,a4paper]{article}

\usepackage[left=2.45cm, top=2.45cm,bottom=2.45cm,right=2.45cm]{geometry}
\usepackage{mathtools,amssymb,amsthm,mathrsfs,calc,graphicx,xcolor,cleveref,dsfont,tikz,pgfplots}
\usepackage[british]{babel}
\usepackage{amsfonts}              
\usepackage{url}
\usepackage{appendix}

\usepackage{pgfplots}
\pgfplotsset{compat=1.18}
\usepgfplotslibrary{groupplots}

\numberwithin{equation}{section}
\numberwithin{figure}{section}

\newtheorem {theorem}{Theorem}[section]
\newtheorem {proposition}[theorem]{Proposition}
\newtheorem {lemma}[theorem]{Lemma}
\newtheorem {corollary}[theorem]{Corollary}

{\theoremstyle{definition}

}
{\theoremstyle{theorem}
\newtheorem {remark}[theorem]{Remark}

}

\def\ba{\begin{array}}
\def\ea{\end{array}}
\def\bea{\begin{eqnarray} \label}
\def\eea{\end{eqnarray}}
\def\be{\begin{equation} \label}
\def\ee{\end{equation}}
\def\bit{\begin{itemize}}
\def\eit{\end{itemize}}
\def\ben{\begin{enumerate}}
\def\een{\end{enumerate}}

\def\EE{\mathbb{E}}
\def\FF{\mathbb{F}}

\def\MM{\mathbb{M}}
\def\NN{\mathbb{N}}
\def\PP{\mathbb{P}}
\def\QQ{\mathbb{Q}}
\def\RR{\mathbb{R}}
\def\RRd1{\mathbb{R}^{d+1}}
\def\SS{\mathbb{S}}
\def\SSd{\mathbb{S}^d}
\def\TT{\mathbb{T}}
\def\WW{\mathbb{W}}
\def\YY{\mathbb{Y}}

\def\t{\tau}

\def\bE{\mathbf{E}}

\def\bP{\mathbf{P}}

\def\cA{\mathcal{A}}
\def\cB{\mathcal{B}}

\def\cD{\mathcal{D}}

\def\cF{\mathcal{F}}

\def\cH{\mathcal{H}}

\def\cT{\mathcal{T}}

\def\cX{\mathcal{X}}

\def\sI{\mathscr{I}}
\def\sL{\mathscr{L}}
\def\sN{\mathscr{N}}

\def\dint{\textup{d}}

\def\GP{\textup{GP}}
\def\Leb{\textup{Leb}}

\makeatletter
\let\@fnsymbol\@alph
\makeatother
\pgfplotsset{compat=1.18} 

\begin{document}

\title{\bfseries Lengths and incidences in Poisson hypersphere\\ and spherical splitting tessellations}

\author{Daniel Hug\footnotemark[1]\;\; and Christoph Th\"ale\footnotemark[2]}

\date{}
\renewcommand{\thefootnote}{\fnsymbol{footnote}}
\footnotetext[1]{Karlsruhe Institute of Technology (KIT), Germany. Email: daniel.hug@kit.edu}

\footnotetext[2]{
Ruhr University Bochum, Germany. Email: christoph.thaele@rub.de}

\maketitle

\begin{abstract}
\noindent
We investigate distributional properties of two random tessellation models on the $d$-dimen\-sional unit sphere. First, we analyze the Poisson hypersphere tessellation generated by a Poisson process on the space of  hyperspheres. We derive an explicit formula for the length distribution of its typical edge. Second, we turn to spherical splitting tessellations, which form a natural class of random tessellations driven by a geometry-dependent Markovian split dynamics. We obtain an exact expression for the length distribution of the typical maximal segment. Unlike in the Poisson model, these maximal segments may exhibit internal incidences. For \(d=2\), we explicitly compute the probability that the typical maximal segment has  a given number of such interior incidences. Some of our results rely on a new Mecke-type formula adapted to the spherical splitting process.
\end{abstract}

{\footnotesize
\tableofcontents
}

\section{Introduction}

Random tessellations arise in many areas of stochastic geometry, spatial statistics, and applied probability.  Roughly speaking, a tessellation of a space is a partition of that space into random cells by means of randomly generated dividing elements.  In Euclidean space, classical examples include the Poisson hyperplane tessellation, where the hyperplanes of an underlying Poisson process cut the space into convex polytopes, the Poisson–Voronoi tessellation, generated by partitioning space into nearest‐neighbour cells around the points of a Poisson point process, and the so-called STIT (stable under iteration) tessellations, obtained by a recursive splitting mechanism that yields a rich hierarchical structure of cells. These tessellation models have been studied intensively, both for their intrinsic mathematical interest and for diverse applications ranging from modeling materials microstructures and geological fracture networks to biological cellular tissues, telecommunication coverage zones and machine learning. For general references on random tessellations we refer to \cite[Chapter 9]{CSKM}, \cite[Chapter 10]{SW} and the monographs \cite{HugSchneiderBook,NTWbook,Okabe} on Poisson hyperplane tessellations, STIT tessellations and Voronoi tessellations, respectively.

\begin{figure}[t]
\begin{center}
\includegraphics[
  trim = 75mm 75mm 75mm 75mm,
  clip,
  width=0.4\columnwidth
]{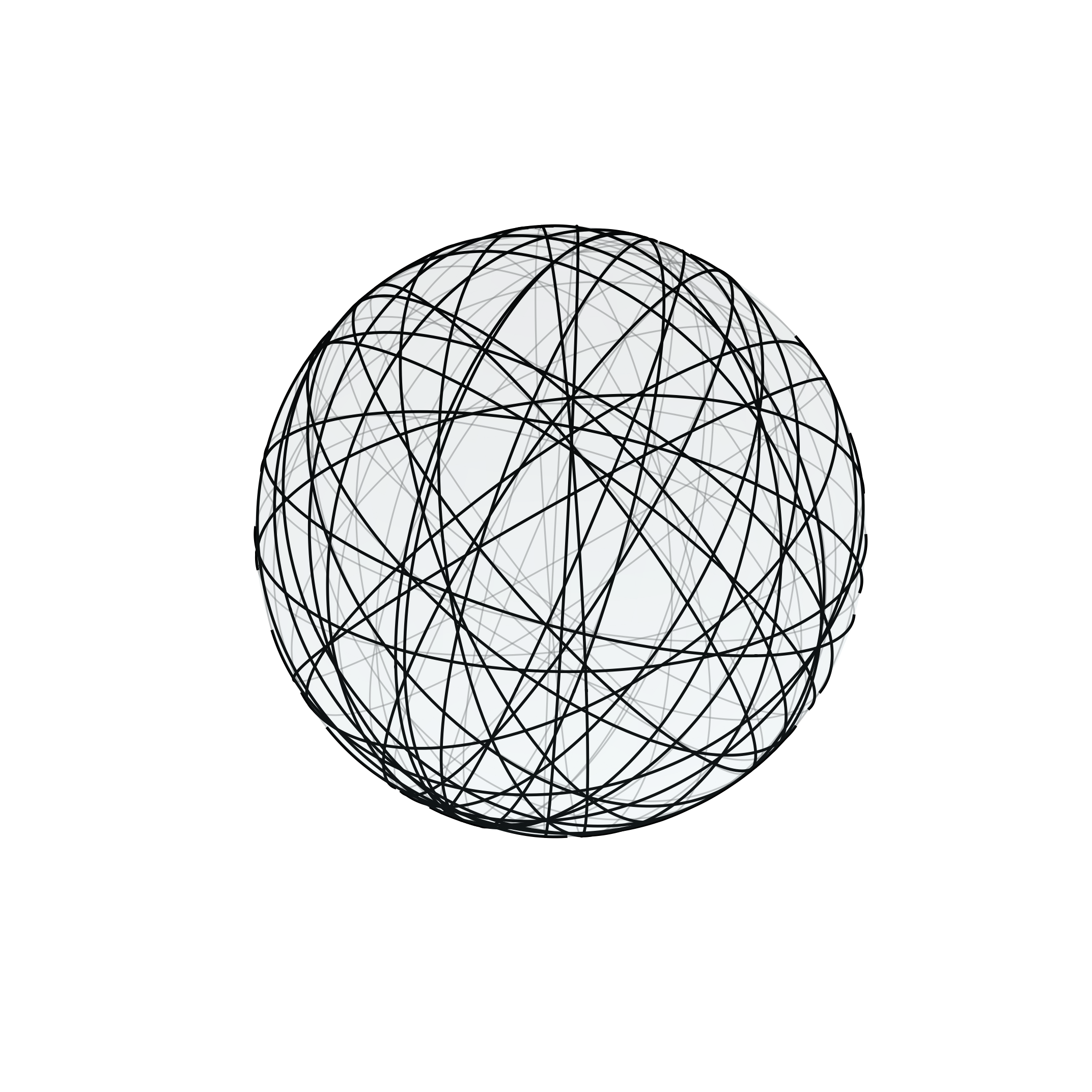}\qquad
\includegraphics[
  trim = 75mm 75mm 75mm 75mm,
  clip,
  width=0.4\columnwidth
]{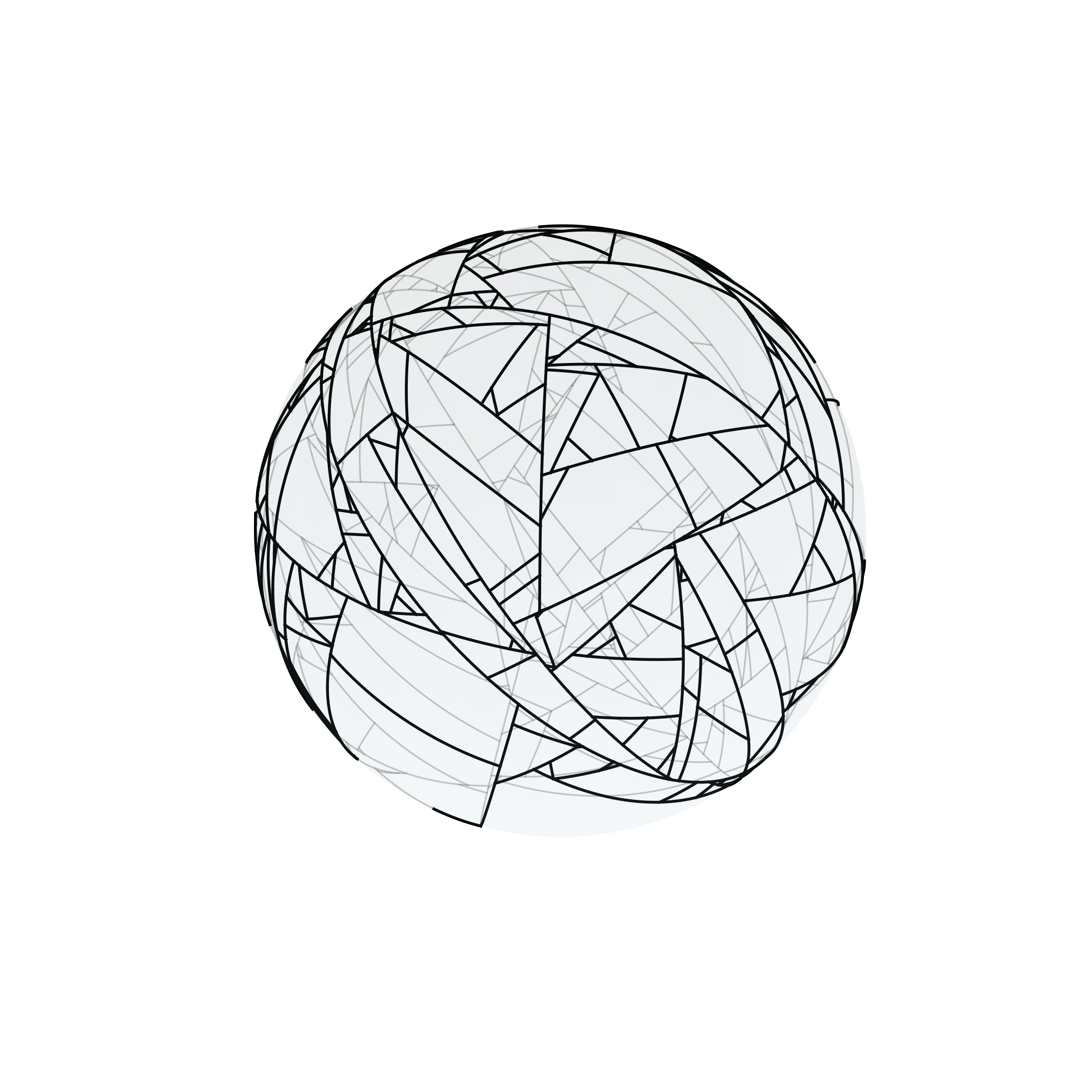}
\end{center}
\caption{Illustration of a Poisson hypersphere tessellation (left) and a splitting tessellation (right) on $\SS^2$.}
\label{fig:PoissonGreatHypersphere}
\end{figure}

In a Poisson hyperplane tessellation, hyperplanes arise as the atoms of a stationary (and often isotropic) Poisson process on the space of affine hyperplanes in \(\RR^d\). The resulting partition has been the subject of many classical results. For instance, the mean number of faces of various dimensions or the mean typical cell volume have all been derived in closed form, see the monograph \cite[Theorem 3.2.1]{HugSchneiderBook}. Also the typical edge-length distribution has been identified as the exponential distribution \cite[Example 1.5]{BaumstarkLast}.  In contrast, STIT tessellations, where cells are split one at a time by random hyperplanes chosen with a rate proportional to their mean width (in the isotropic case), exhibit temporal and spatial dependencies that complicate their analysis.  Nevertheless,  in these models remarkable progress has been made on mean values \cite{MNW08,MNW11,NNTW,NW05,NagelWeissPlane,STBernoulli}, variance asymptotics \cite{STSTITPlane,STSecondOrder}, and limit theorems \cite{STSTITPlane,STSTITHigher,STITLimits}
for cell sizes, face counts, and a full spectrum of geometric characteristics ranging from one-dimensional segment lengths to $(d-1)$-dimensional facet contents.

One of the most fundamental classes of objects in any planar or spatial tessellation is that of edges or segments.  In a face‐to‐face tessellation such as the Poisson hyperplane or Voronoi tessellation, the edges are precisely the line segments that arise as intersections of sufficiently many adjacent cells. In contrast, non‐face‐to‐face tessellations, such as STIT tessellations, do not necessarily align cell boundaries globally, and so edges may terminate at so-called $\pi$‐junctions which are contained in the relative interior of another line segment.  In this setting, the notion of an edge splits into several subtypes as motivated and explained  in detail in \cite{CowanThaele,WeissCowan}. Relevant in our context are the concept of a genuine edge and that of a maximal segment. By an edge we understand a line segment, which is bounded by two vertices and which does not contain further tessellation vertices in its relative interior. On the other side of the spectrum, the maximal segments of a tessellation are those line segments which cannot be extended further by collinear pieces. In a STIT tessellation, maximal segments are precisely the bounded line segments that are born during the recursive cell-division construction and are subsequently subdivided by later cell divisions. In a sense, the maximal segments encode the history of the recursive splitting and form the backbone of the STIT structure. A detailed analysis of the maximal segments of STIT tessellations in $\RR^d$ has been the content of \cite{MNW11,NNTW,NTWbook,Thaele2009,Thaele2010}. In a Poisson hyperplane tessellation, however, there is no analogous intrinsic notion of a bounded maximal segment.

While the theory of random tessellations in $d$-dimensional Euclidean spaces is well-developed, much less is known about random tessellations on the  $d$-dimensional unit sphere \(\SS^d\).  This is somewhat surprising, since the sphere is the most basic compact space of constant curvature and provides a natural setting for tessellations generated by directional, geodesic or great-subsphere data.  At the same time, spherical tessellations are not merely compact versions of their Euclidean analogues.  The absence of translations and dilations, the boundedness of all geodesic segments, and the global topology of \(\SS^d\) induce structural effects that have no direct Euclidean counterpart.  Thus spherical tessellations form a natural test case for understanding how curvature, compactness and topology affect classical models from stochastic geometry.

Mean values, second-order parameters of their total face  content, and the so-called Kendall problem have been studied for spherical analogues of Poisson hyperplane tessellations, in which hyperplanes are replaced by hyperspheres, in  \cite{ArbeiterZaehleMosaics,HeroldHugThaele,HugReichenbacher,HugSchneider2016,HugThaele18,MilesSphere} (see the left panel of Figure \ref{fig:PoissonGreatHypersphere} for an illustration).  More recently, spherical analogues of STIT tessellations, called spherical splitting tessellations, have been introduced via a geometry-driven Markovian cell-splitting dynamics, in which cells are successively divided by random hyperspheres according to their spherical geometry  (see the right panel of Figure \ref{fig:PoissonGreatHypersphere} for an illustration).  Initial results for these models include mean values and a detailed second-order analysis of the total surface content \cite[Sections 5.2-3]{HugThaele18}.  However, distributional information on their lower-dimensional building blocks is still rather limited.  In particular, precise  formulas for (the distribution of) one-dimensional features such as the length of the typical edge in the Poisson hypersphere tessellation and the length of the typical maximal segment in the spherical splitting tessellation have not been available.  These quantities are natural spherical counterparts of classical edge and segment characteristics in Euclidean tessellation theory, but their behavior is modified by compactness and by the recursive history encoded in the splitting construction.

The contribution of the present paper is to provide explicit distributional
information for the one-dimensional building blocks of spherical tessellations.
First, for a Poisson hypersphere tessellation process on \(\SS^d\), we determine the
distribution of the length \(\overline{\sL}_t\) of the typical edge  at time $t$, for an 
arbitrary regular directional distribution. In the isotropic case, we thus obtain 
a closed formula with two atoms, at \(\pi\) and \(2\pi\), and an
absolutely continuous part on \((0,\pi)\). Second, using the connection between
spherical splitting tessellations and Poisson hypersphere tessellations, we
derive the corresponding distribution of the length \(\sL_t\) of the typical
maximal segment in a spherical splitting tessellation. Again, the isotropic case leads to a completely explicit expression, consisting
of two atoms, at \(\pi\) and \(2\pi\), and a density on \((0,\pi)\). We also
investigate the dependence of these quantities on the dimension. In particular,
for fixed \(t>0\), we show that as \(d\to\infty\) the complete length
distribution of the typical maximal segment converges to the corresponding
length distribution of the typical edge in the Poisson hypersphere
tessellation, and we describe the dimension dependence of the density and the
two point masses in more detail. As \(t\to\infty\) and after suitable
rescaling, in both spherical tessellation models the length distributions will
be shown to converge to the corresponding distributions in Euclidean space.

A further main contribution is a new Mecke-type formula for the spherical splitting
process. This formula describes the typical maximal segment together with the
future splitting processes in the two daughter cells created at its birth time.
It provides the technical tool needed to go beyond length distributions and to
analyze the internal incidence structure of maximal segments. In dimension
\(d=2\), and for the isotropic model, we use it to compute the expectation and the exact probability that the number \(\sI_t\) of internal incidences of the
typical maximal segment equals $n\in\{0,1,2,\ldots\}$. The resulting formula is explicit for every
\(n\) and involves only one integral over the unit interval
and the lower incomplete gamma function. As \(t\to\infty\), these probabilities
converge to the corresponding distribution for the typical \(I\)-segment in a
stationary STIT tessellation in the Euclidean plane, thereby making precise how the spherical
model approaches its flat counterpart in the large-time regime.

\medspace

The remaining parts of the paper are structured as follows. In Section \ref{sec:Prelim} we
collect the notation and some basic facts about typical objects associated with spherical random polytopes. In Section \ref{sec:hsst} we briefly introduce Poisson hypersphere tessellations (equivalently, a Poisson hypersphere tessellation process) and rephrase four equivalent descriptions of the spherical splitting tessellation process. In Section \ref{sec:PoissonHyperspheres} we determine the length distribution of the typical edge in a Poisson hypersphere tessellation, while Section \ref{sec:Splitting} treats the corresponding length distribution of the typical maximal segment in a spherical splitting tessellation. Section \ref{sec:HighDimensions} is devoted to the high-dimensional behaviour of these distributions. In Section \ref{sec:Mecke} we derive a Mecke-type formula for maximal segments of spherical splitting tessellations, which is then used in Section \ref{sec:SphericalMaximalSegments} to investigate the number of internal incidences of the typical maximal segment in the isotropic two-dimensional case.

\section{Preliminaries}\label{sec:Prelim}

\subsection{Frequently used notation}

For \(d\geq 2\) we let \(\SSd\) be the \(d\)-dimensional unit sphere in
\(\RR^{d+1}\), and we write \(\cB(\SSd)\) for its Borel \(\sigma\)-field.  By
\(\cH^d\) we denote the \(d\)-dimensional Hausdorff measure restricted to
\(\SSd\).  Thus
\[
\beta_d:=\cH^d(\SSd)
=
\frac{2\pi^{(d+1)/2}}{\Gamma((d+1)/2)} ,
\]
where \(\Gamma(\,\cdot\,)\) is the usual Gamma function.  We denote by $\sigma_d:=\beta_d^{-1}\cH^d$
the normalized spherical Lebesgue measure on \(\SSd\).  More generally,
\(\cH^s\), \(s\geq0\), denotes the \(s\)-dimensional Hausdorff measure.  The
geodesic distance on \(\SSd\) is denoted by \(d_{\SSd}(\,\cdot\,,\,\cdot\,)\).

For \(k\in\{0,1,\ldots,d\}\), we write $G(d+1,k+1)$ for the set of all $(k+1)$-dimensional linear subspaces of $\RR^{d+1}$,  denote by \(\SS_k\) the space of
\(k\)-dimensional great subspheres of \(\SSd\), that is,
\[
\SS_k
:=
\{L\cap\SSd:L\in G(d+1,k+1)\},
\]
that is, $\SS_d=\{\SS^d\}$,
and define
$$
\SS_{0:d} := \bigcup_{k=0}^d\SS_k.
$$
In particular, \(\SS_{d-1}\) is the space of great hyperspheres of \(\SSd\).
If \(S\in\SS_{d-1}\), then \(S\) divides \(\SSd\) into two closed hemispheres,
which we denote by \(S^+\) and \(S^-\).  The choice of the signs is arbitrary
and will never affect the arguments.  If \(A\subset\SSd\), we put
\[
\SS_{d-1}[A]
:=
\{S\in\SS_{d-1}:S\cap A\neq\varnothing\}.
\]
For a measure \(\mu\) and a measurable set \(B\), the restriction of \(\mu\)
to \(B\) is denoted by \(\mu\llcorner B\).

For \(u\in\SSd\), the intersection $u^\perp\cap\SSd$ is 
 the great hypersphere with normal direction \(\pm u\).  The map
\(u\mapsto u^\perp\cap\SSd\) identifies antipodal directions.  Thus a Borel
probability measure \(\kappa\) on \(\SS_{d-1}\) can be represented by a
symmetric Borel probability measure \(\kappa_o\) on \(\SSd\), uniquely
determined by
\begin{equation}\label{eq:kappa0}
\kappa(\,\cdot\,)
=
\int_{\SSd}
{\bf 1}\{u^\perp\cap\SSd\in\,\cdot\,\}\,\kappa_o(\dint u).
\end{equation}
Throughout the paper, \(\kappa\) is assumed to be regular, in the sense that
\[
\kappa\big(\{S\in\SS_{d-1}:e\in S\}\big)=0,
\qquad e\in\SSd .
\]

Let \(\nu_d\) denote the rotation-invariant Haar probability
measure on \(G(d+1,d)\).  If \(z\in\SSd\), we write
\[
\Theta_z:=\{\varrho\in {\rm SO}_{d+1}:\varrho n=z\},
\qquad n:=(1,0,\ldots,0),
\]
for the set of (orientation preserving) rotations which map the `north pole' \(n\) to \(z\).

If \(\MM\) is a separable metric space with its Borel $\sigma$-field and \(x\in\MM\), then \(\delta_x\) denotes the
Dirac measure concentrated at \(x\).  We write \(\sN(\MM)\) for the space of
locally finite counting measures on \(\MM\), equipped with its usual
\(\sigma\)-field.  Lebesgue measure on \((0,\infty)\) and on \((-\infty,0)\)
will be denoted by \(\Leb_+\) and \(\Leb_-\), respectively.

We shall use the lower incomplete gamma function
\[
\gamma(a,x):=\int_0^x u^{a-1}e^{-u}\,\dint u,
\qquad a>0,\ x\geq0,
\]
and the exponential integral
\[
E_1(x):=\int_x^\infty \frac{e^{-u}}{u}\,\dint u,
\qquad x>0.
\]
The Euler-Mascheroni constant is denoted by \(C_{\operatorname{EM}}\approx 0,57721\).

\subsection{Tessellations and typical objects}

By a spherical polytope we understand the intersection of $\SS^d$ with a polyhedral cone in $\RRd1$, that is, with a finite intersection of closed halfspaces whose bounding hyperplanes pass through the origin of $\RRd1$. For the empty family of hyperplanes, we obtain $\SSd$ itself. By $\PP^d$ we denote the space of spherical polytopes. 
A tessellation $T$ of $\SSd$ is a finite subset of $\PP^d$ (whose elements are called cells) such that
\begin{itemize}
\item[(i)] each $c\in T$ has non-empty interior,
\item[(ii)] $\bigcup_{c\in T}c=\SSd$,
\item[(iii)] any two distinct cells $c_1,c_2\in T$ have disjoint interiors.
\end{itemize}
The space of tessellations on $\SSd$ will be denoted by $\TT^d$. It carries a natural $\sigma$-field $\cT^d$ and the corresponding measurability questions, including the measurability of the cell and face functionals used below, have been treated in detail in \cite{HugThaele18} (and in the references given there). We shall use this measurable structure throughout and will not repeat these technical points here. For a tessellation $T\in\TT^d$ and $k\in\{0,1,\ldots,d-1\}$ we let $\cF_k(T)$ be the set of all $k$-dimensional faces of cells of $T$. Note that each such $k$-face appears only once in $\cF_k(T)$ even if it is a $k$-face of more than one cell of $T$.

Let $0$ denote the origin of $\RRd1$. In the following, we consider a center function $z:\PP^d\to \SS^d\cup \{0\}$ which assigns to each $c\in \PP^d$ a point $z(c)\in \SS^d\cup \{0\}$ in a rotation covariant way, that is, we require that $z(\varrho c)=\varrho z(c)$ for all rotations $\varrho$ of $\SS^d$. If $c\in \SS_{0:d}$, then we define $z(c):=0$. For $c\in\PP^d\setminus\SS_{0:d}$, a natural choice for $z$ is $z(c):=z_0(\left(c\cap (-c)\right)^\perp\cap c)$, where $z_0(x)$ denotes the unique center of the circumball of a spherical convex body contained in an open hemisphere. Note that the center of the circumball is not uniquely determined if $x$ contains antipodal points, see \cite[Remark 4.5]{HugReichenbacher}. If $c\in \PP^d$ is contained in an open hemisphere, then $(c\cap(-c))^\perp=\RR^{d+1}$ and $z(c)=z_0(c)$. 

By a random tessellation $\tau$ on $\SSd$ we understand a random element of the measurable space $(\TT^d,\cT^d)$ defined on some probability space $(\Omega,\cA,\bP)$. A random tessellation on $\SSd$ is called isotropic, if its distribution is invariant under all rotations of $\SSd$.

Let $\tau$ be a random tessellation of $\SSd$ and $\cX=\cX(\tau)$ be a (finite) random process of (potentially lower-dimensional) spherical random polytopes determined by $\tau$. 
 For $z\in\SSd$ recall the definition of $\Theta_z$, define $\Theta_0:={\rm SO}_{d+1}$ and let $\nu_0$ be the invariant Haar probability measure on $\Theta_0$. 
If $n:=(1,0,\ldots,0)\in\SSd$ denotes the north pole of $\SSd$, we let $\bP_{\cX}^n$ be the probability measure on $\PP^d$ defined by
\begin{align*}
\bP_{\cX}^n(\,\cdot\,) = \frac{1}{\EE|\cX|}\EE\sum_{x\in\cX}\int_{\Theta_{z(x)}}{\bf 1}\{\varrho^{-1} x \in\,\cdot\,\}\,\nu_{z(x)}(\dint\varrho),
\end{align*}
where $\nu_z:=\nu_n\circ\varrho_z^{-1}$, for $z\in \SS^d\cup \{0\}$, $\varrho_z$ is an arbitrary element of $\Theta_z$ (the definition can be shown to be independent of the particular choice of $\varrho_z$) and $\nu_n$ is the rotation-invariant Haar probability measure on $\Theta_n$.
A spherical random polytope with distribution $\bP_{\cX}^n$ is called the typical object of class $\cX$. For example, if $\cX$ is the class of cells of $\tau$, the typical object of this class is the so-called typical cell of $\tau$, and if $\cX$ is the class of tessellation edges, the typical object is the so-called typical edge of $\tau$.

\section{Hypersphere and spherical splitting tessellations}
\label{sec:hsst}
\subsection{Poisson hypersphere tessellations}

We fix a space dimension $d\geq 2$ and an intensity parameter $t>0$. Recall that $\SS_{d-1}$ is the space of hyperspheres of $\SSd$. Let $\kappa$ be a Borel probability measure on $\SS_{d-1}$, which we assume to be regular. For any such probability measure $\kappa$ there exists a uniquely determined symmetric probability measure $\kappa_o$ on $\SSd$ such that \eqref{eq:kappa0} holds.  The measure $\kappa_o$ is regular in the sense that $\kappa_o(S)=0$ for each $S\in \SS_{d-1}$; in particular, $\kappa_o$ is diffuse. 
Let $\eta_t$ be a (simple) Poisson point process on $\SSd$ with intensity measure $t\kappa_o$, $t>0$. In other words, $\eta_t$ is a collection of random points $\{U_1,\ldots,U_N\}$, where the random points $U_1,U_2,\ldots$ are independent and identically distributed on $\SSd$ according to $\kappa_o$ and the random variable $N$ is Poisson distributed with mean $t$ and $N,U_1,U_2,\ldots$ are independent. We consider the map $F:\SSd\to \SS_{d-1}$, which assigns to a unit vector $u\in\SSd$ the intersection of $\SS^d$ with the linear subspace $u^\perp$ totally orthogonal to $u$. This gives rise to the random closed set
$$
 \bigcup_{u\in\eta_t}F(u) = \bigcup_{u\in\eta_t}  u^\perp \cap \SS^d,
$$
which partitions the unit sphere $\SSd$ into an almost surely finite collection of spherical random polytopes, see Figure \ref{fig:PoissonGreatHypersphere}. By the mapping property of Poisson point processes, $F(\eta_t)$ is a Poisson point process on $\SS_{d-1}$, whose intensity measure is the image measure $t\kappa_o\circ F^{-1}$ of $t\kappa_o$ under the mapping $F$. However, the definition of $\kappa_o$ implies that $t\kappa_o\circ F^{-1}=t\kappa$. We call the induced partition of $\SSd$ the Poisson  hypersphere tessellation of intensity $t$ and directional distribution $\kappa$. It will be denoted by $X_t$, suppressing thereby in the notation the dependence on $\kappa$ which is considered to be fixed. It is shown in \cite[Section 7.3]{HugThaele18} that $X_t$ is equal in distribution to a Poisson great hypersphere tessellation process  with direction distribution $\kappa$ at time $t$. With $\eta_0=0$ (for $t=0$) we consistently obtain $X_0=\{\SS^d\}$. 

Let $k\in\{0,1,\ldots,d-1\}$ and recall the definition of $\cF_k(X_t)$. We define the following two measures:
$$
\overline{\cF}_{X_t}^{(k)} := \sum_{f\in\cF_k(X_t)}\delta_f\qquad\text{and}\qquad \overline{\FF}_{X_t}^{(k)} := \EE\overline{\cF}_{X_t}^{(k)}.
$$
Since later we are mostly interested in tessellation edges, we shall use the shorthand notation $\overline{\cF}_{X_t}$ and $\overline{\FF}_{X_t}$ for $\overline{\cF}_{X_t}^{(1)}$ and $\overline{\FF}_{X_t}^{(1)}$, respectively.

\subsection{Spherical splitting tessellations}

In this section we provide four equivalent descriptions of the spherical splitting tessellation process. The first description is the direct recursive construction.  The second
rewrites the same dynamics as a pure-jump Markov process.  The third gives the
corresponding explicit path distribution, which will be used for conditioning
arguments.  These descriptions were already used in \cite{HugThaele18}. The fourth one, which is inspired by the approach in \cite{NNTW}, realizes the same process from a single underlying Poisson
point process and is useful for Mecke-type identities developed in Section \ref{sec:Mecke}.

\subsubsection{Algorithmic description}

At time \(t=0\) the construction starts from the trivial tessellation
\(Y_0=\{\SS^d\}\).  The initial cell \(\SS^d\) is assigned an exponential
lifetime with parameter \(\kappa(\SS_{d-1}[\SS^d])=1\).  When this lifetime
expires, a hypersphere \(S\in\SS_{d-1}\) is chosen according to \(\kappa\), and
\(\SS^d\) is split into the two cells
\[
c^+:=\SS^d\cap S^+
\qquad\text{and}\qquad
c^-:=\SS^d\cap S^- .
\]
The maximal polytope born at this time is \(S\), and its birth time is denoted by \(\beta(S)\). These two new cells are `born' simultaneously at the split time \(\beta(S)\) and are each assigned (conditionally) independent exponential lifetimes: \(c^+\) receives a lifetime with parameter \(\kappa\bigl(\SS_{d-1}[c^+]\bigr)\), and \(c^-\) receives one with parameter \(\kappa\bigl(\SS_{d-1}[c^-]\bigr)\).  Crucially, no further splitting of \(c\) occurs. Once a cell has split, only its descendants split in the future.

Recursively, whenever a cell \(c\) with lifetime parameter \(\kappa\bigl(\SS_{d-1}[c]\bigr)\) dies, we choose a new  hypersphere \(S\) from the normalized measure
\[
\frac{\kappa\bigl(\cdot \,\cap\, \SS_{d-1}[c]\bigr)}{\kappa\bigl(\SS_{d-1}[c]\bigr)},
\]
split \(c\) into \(c^+ = c\cap S^+\) and \(c^- = c\cap S^-\) at time \(\beta(c\cap S)\), and assign each child its own independent exponential lifetime with parameters \(\kappa\bigl(\SS_{d-1}[c^+]\bigr)\) and \(\kappa\bigl(\SS_{d-1}[c^-]\bigr)\), respectively.  In this construction it is assumed that all random elements (lifetimes and  hyperspheres) are mutually independent.

For any fixed time $t\geq 0$ we denote by $M_t$ the random collection of pairs $(c\cap S,\beta(c\cap S))$ of  spherical $(d-1)$-polytopes $c\cap S$ that have been constructed until time $t$ when a cell $c$ was split by a  hypersphere $S$, together with their birth times. For $t\geq 0$ we denote by $Y_t$ the closures of the complement of $\bigcup_{(p,s)\in M_t}p$. In particular, $Y_t$ is a spherical random tessellation, the so-called splitting tessellation with time parameter $t$ and directional distribution $\kappa$.

\subsubsection{Description as a Markov process}

Our next description is based on the idea that $(Y_t)_{t\geq 0}$ is a continuous time pure-jump Markov process on the space $\TT^d$ of tessellations on $\SS^d$. It is driven by a regular probability measure $\kappa$ on $\SS_{d-1}$. Its initial state is $Y_0:=\{\SSd\}$ and its generator $\cA$ is given by
$$
(\cA f)(T) := \sum_{c\in T}\int_{\SS_{d-1}[c]}[f(\oslash_{c,S,T})-f(T)]\,\kappa(\dint S),\qquad T\in\TT^d.
$$
Here, $f:\TT^d\to\RR$ is a bounded measurable function and for $T\in\TT^d$, $c\in T$ and $S\in\SS_{d-1}[c]$ the tessellation $\oslash_{c,S,T}$ arises by splitting the cell $c$ by the  hypersphere $S$:
$$
\oslash_{c,S,T} := (T\setminus\{c\})\cup\{c\cap S^+,c\cap S^-\}.
$$
This generator is the infinitesimal form of the preceding recursive
construction.  Indeed, when the current tessellation is \(T\), a cell
\(c\in T\) is split with rate \(\kappa(\SS_{d-1}[c])\), and, conditional on
this event, the splitting hypersphere has distribution
\[
{\kappa(\,\cdot\,\cap\SS_{d-1}[c])\over
	\kappa(\SS_{d-1}[c])}.
\]
Equivalently, the transition measure from \(T\) is
\[
\sum_{c\in T}\int_{\SS_{d-1}[c]}
\delta_{\oslash_{c,S,T}}\,\kappa(\dint S),
\]
which is precisely the jump kernel appearing in \(\cA\).

\subsubsection{Description of the distribution}

In \cite{HugThaele18} the explicit distribution of a spherical splitting tessellation process has been described. To rephrase the result, we need to introduce some further notation. As above, we let $\kappa$ be a regular probability measure on $\SS_{d-1}$. For $T\in\TT^d$ we define a measure $\phi(T;\,\cdot\,)$ on $\PP^d\times\SS_{d-1}$ by
$$
\phi(T;\,\cdot\,) := \sum_{c\in T}\delta_c\otimes\kappa\llcorner\SS_{d-1}[c],
$$
where $\kappa\llcorner\SS_{d-1}[c]$ stands for the restriction of $\kappa$ to $\SS_{d-1}[c]$. We denote by $\phi(T):=\phi(T;\PP^d\times\SS_{d-1})$ the total mass of this measure.

For \(T_a\in\TT^d\), \(0\leq a\leq b\), \(n\in\NN_0\),
\(a<s_1<\ldots<s_n<b\), \(c_1,\ldots,c_n\in\PP^d\) and
\(S_i\in\SS_{d-1}[c_i]\), \(1\leq i\leq n\), we denote by
\[
\cD\big(T_a;[a,b];(s_i,c_i,S_i)_{1\leq i\leq n}\big)
\]
the set of c\`adl\`ag paths \((\Upsilon_u)_{u\in[a,b]}\) in \(\TT^d\)
such that \(\Upsilon_a=T_a\), the path is constant except for jumps at
\(s_1,\ldots,s_n\), and
\[
\Upsilon_{s_i}
=
\oslash_{c_i,S_i,\Upsilon_{s_i-}},
\qquad
c_i\in\Upsilon_{s_i-},\quad S_i\in\SS_{d-1}[c_i],
\qquad i=1,\ldots,n,
\]
where for $s>a$, $\Upsilon_{s-}$ denotes the left limit of $\Upsilon_{u}$ as $u\uparrow s$. This allows us to describe the distribution of the splitting tessellation process with initial tessellation $T_a\in\TT^d$ within a time interval $[a,b]$ as follows:
\begin{align*}
	\bP\big((Y_u^{(T_a)})_{u\in[a,b]}\in\,\cdot\,\big)
	&=
	\sum_{n=0}^\infty
	\int_{a<s_1<\ldots<s_n<b}
	\int\cdots\int
	\exp\Big\{-\int_a^b\phi(\Upsilon_u)\,\dint u\Big\}\\
	&\qquad\qquad\times{\bf 1}\Big\{
	(\Upsilon_u)_{u\in[a,b]}
	\in
	\cD\big(T_a;[a,b];(s_i,c_i,S_i)_{1\leq i\leq n}\big)
	\Big\} \\
	&\qquad\qquad\times
	{\bf 1}\{(\Upsilon_u)_{u\in[a,b]}\in\,\cdot\,\}
	\prod_{i=1}^n
	\phi(\Upsilon_{s_i-};\dint(c_i,S_i))
	\,\dint s_1\cdots\dint s_n .
\end{align*}
Here the multiple integral with respect to
\(\prod_{\ell=1}^n\phi(\Upsilon_{s_\ell-};\dint(c_\ell,S_\ell))\) is read
iteratively along the path: after each jump the path is updated, and the next
factor is evaluated at the left limit of the new path. If $n=0$, then the integrations are omitted and $\Upsilon_u=T_a$ for $u\in [a,b]$.

\subsubsection{Description by a Poisson point process}

Finally, we provide yet another description of the spherical splitting tessellation process, which is based on a single Poisson point process on a rather intricate state space. It is inspired by the global construction from \cite{MNW08} of STIT tessellations in $\RR^d$, see also \cite[Chapter 9]{NTWbook}. To present it, we let again $\kappa$ be a regular Borel probability measure on $\SS_{d-1}$ and $\Pi$ be a Poisson point process on $\SS_{d-1}\times(0,\infty)$ with intensity measure $\kappa\otimes\Leb_+$, where $\Leb_+$ denotes the Lebesgue measure on $(0,\infty)$. The distribution of $\Pi$ is denoted by $\bP_\Pi$. Next, we consider a Poisson point process $\Sigma$ on $\SSd\times(-\infty,0)\times\sN(\SS_{d-1}\times(0,\infty))$ with intensity measure $\sigma_d\otimes\Leb_-\otimes\bP_\Pi$, where $\Leb_-$ is the Lebesgue measure on $(-\infty,0)$ and $\sN(\SS_{d-1}\times(0,\infty))$ stands for the space of locally-finite  counting measures on $\SS_{d-1}\times(0,\infty)$. For a full-dimensional spherical polytope $c\in\PP^d$ we let $(X(c),R(c),\Psi(c))\in\Sigma$ be the almost surely uniquely determined element of $\Sigma$ such that
$$
X(c)\in c\qquad\text{and}\qquad R(c)=\max\{r\in(-\infty,0):(x,r,\psi)\in\Sigma,x\in c\}.
$$
We use the Poisson point process $\Sigma$ to describe the following construction of a pure-jump process $(\widetilde{Y}_t)_{t\geq 0}$ on the space of spherical tessellations together with a pure-jump process $(\widetilde{M}_t)_{t\geq 0}$ of marked spherical $(d-1)$-dimensional polytopes with marks in $(0,\infty)$.
\begin{itemize}
\item[(i)] Select the triplet $(X(\SSd),R(\SSd),\Psi(\SSd))$ and then the pair $(S,s)\in\Psi(\SSd)$ such that $s=\min\{s'>0:(S',s')\in\Psi(\SS^d)\}$. Put $\widetilde{Y}_t:=\{\SSd\}$ and $\widetilde{M}_t:=\varnothing$ for $0\leq t<s$.
\item[(ii)] At time $s$, the two processes jump as follows: $\widetilde{Y}_s:=\{S^+,S^-\}$, $\widetilde{M}_s:=\{(S,s)\}$ and we say that the two cells $S^+$ and $S^-$ are born at time $s$ and denote their birth times by $\beta(S^\pm)=s$.
\item[(iii)] Whenever a cell $c$ is born and has birth time $\beta(c)$, first select the triplet $(X(c),R(c),\Psi(c))$ and then a pair $(S,s)\in\Psi(c)$ such that $s=\min\{s'>\beta(c):(S',s')\in\Psi(c),S'\in\SS_{d-1}[c]\}$. At any such time $s$ the two processes jump as follows: $\widetilde{Y}_s:=(\widetilde{Y}_{s-}\setminus\{c\})\cup\{c\cap S^+,c\cap S^-\}=\oslash_{c,S,Y_{s-}}$ and $\widetilde{M}_s:=\widetilde{M}_{s-}\cup\{(c\cap S,s)\}$, where $\widetilde{Y}_{s-}$ and $\widetilde{M}_{s-}$ denote the left limits at $s$.
\end{itemize}

\begin{lemma}
The random processes $(\widetilde{Y}_t)_{t\geq 0}$ and $(\widetilde{M}_t)_{t\geq 0}$ have the same distributions as the random processes $(Y_t)_{t\geq 0}$ and $(M_t)_{t\geq 0}$ introduced earlier in this section.
\end{lemma}
\begin{proof}
To prove the claim, we need to verify that the distribution of the  hyperspheres used to divide cells have distribution $\kappa$ and that the time until a cell $c$ gets divided is exponentially distributed with parameter $\kappa(\SS_{d-1}[c])$. While the first property automatically holds by construction, for the second one we let $c$ be some cell with birth time $\beta(c)$, $\Psi(c)$ its associated time-marked Poisson  hypersphere process and $(S,s)\in\Psi(c)$ be as in the description above. Then, for $u>0$, 
\begin{align*}
\bP(s-\beta(c)>u\,|\,\mathfrak{I}_{\beta(c)}) &= \bP(\min\{s'>\beta(c):(S',s')\in\Psi(c),S'\in\SS_{d-1}[c]\}>\beta(c)+u)\\
&=\bP(\forall (S's')\in\Psi(c)\text{ with }s'>\beta(c)\text{ and }S'\in\SS_{d-1}[c]:s'>\beta(c)+u),
\end{align*}
where $\mathfrak{I}_t$ for $t>0$ denotes the $\sigma$-field generated by $(\widetilde{Y}_s)_{0\leq s\leq t}$.
By the thinning property of Poisson point processes,
$$
\widetilde{\Psi}(c):=\{(S's')\in\Psi(c)\text{ with }s'>\beta(c)\text{ and }S'\in\SS_{d-1}[c]\}
$$
is a Poisson point process on $\SS_{d-1}[c]\times(\beta(c),\infty)$ with intensity measure $\kappa\llcorner\SS_{d-1}[c]\otimes\Leb_{\beta(c)+}$, where $\Leb_{\beta(c)+}$ denotes the Lebesgue measure on $(\beta(c),\infty)$. Thus,
\begin{align*}
\bP(s-\beta(c)>u\,|\,\cF_{\beta(c)}) &= \bP(\widetilde{\Psi}(c)(\SS_{d-1}[c]\times(\beta(c),\beta(c)+u])=0)\\
&=\exp\Big(-(\kappa\llcorner\SS_{d-1}[c]\otimes\Leb_{\beta(c)+})(\SS_{d-1}[c]\times(\beta(c),\beta(c)+u])\Big)\\
&=\exp\Big(-\kappa(\SS_{d-1}[c])u\Big).
\end{align*}

Moreover, if two cells have disjoint relative interiors, then the restrictions
of \(\Sigma\) to the corresponding sets of spatial locations are independent.
Since the boundary of a spherical polytope has \(\sigma_d\)-measure zero, this
applies almost surely to distinct cells of the tessellation.  Thus, after a
split, the point-process data used for the two daughter cells are independent,
up to the continuation of the time-marked Poisson process which has already
produced the split.  The latter causes no dependence on the past, because
Poisson processes have independent increments: conditional on the first
admissible point after the birth time, the configuration after this time is
again a Poisson process with the same intensity, independent of the past.
Consequently the two daughter cells evolve independently and according to the
same splitting rule as in the recursive construction. More precisely, the time until cell $c$ gets split by a  hypersphere with distribution $\kappa(\SS_{d-1}[c])^{-1}\kappa\llcorner\SS_{d-1}[c]$ is exponentially distributed with parameter $\kappa(\SS_{d-1}[c])$.
\end{proof}

\section{Edge length distribution in a Poisson  hypersphere tessellation}\label{sec:PoissonHyperspheres}

We let $X_t$ be a Poisson  hypersphere tessellation of $\SSd$ with intensity $t>0$ and regular directional distribution $\kappa$. By this we mean that $X_t$ is obtained from a Poisson hypersphere process with intensity measure $t\kappa$. The goal of this section is to determine the precise distribution of the length $\overline{\sL}_t$ of the typical edge of $X_t$.

To formulate the result, some further notation is necessary. By \(\GP_{d-1}\)
we denote the set of \((d-1)\)-tuples \((u_1,\ldots,u_{d-1})\) such that
\[
u_1^\perp\cap\ldots\cap u_{d-1}^\perp\cap\SSd\in\SS_1 .
\]
Equivalently, \((u_1,\ldots,u_{d-1})\in\GP_{d-1}\) if and only if
\(u_1,\ldots,u_{d-1}\) are linearly independent. We note that
$\kappa_o^{d-1}(\GP_{d-1})=1$. In fact, for $d=2$ there is nothing to show. If $d\ge 3$ and  \(u_1,\ldots,u_{d-1}\) are linearly dependent, then we may assume that 
  \(u_{d-1}\) is in the linear span of \(u_1,\ldots,u_{d-2}\), hence $u_{d-1}$ lies in a great subsphere which has  \(\kappa_o\)-measure zero, by the regularity assumption on \(\kappa\). Thus the assertion follows from Fubini's theorem.

We consider the map \(F_\cap:\GP_{d-1}\to\SS_1\) given by
\[
F_\cap(u_1,\ldots,u_{d-1})
:=
u_1^\perp\cap\ldots\cap u_{d-1}^\perp\cap\SSd .
\]
The image measure of \(\kappa_o^{d-1}\) under \(F_\cap\) is denoted by
\(\kappa_\cap\). For \(S\in\SS_1\), define the probability measure \(\kappa_S\) on \(S\) by
\[
\int_S f(v)\,\kappa_S(\dint v)
=
{1\over 2}\int_{\SS^d}
\sum_{v\in u^\perp\cap S} f(v)\,\kappa_o(\dint u),
\]
for all non-negative measurable \(f:S\to\RR\). Finally, for $S\in\SS_1$, a point $u\in S$ and $0\leq s\leq 2\pi$ we shall write $u^s$ for the set of all points $v\in S$ having (geodesic) distance less than or equal to $s$ from $u$, that is, $u^s:=\{v\in S:d_{\SS^d}(u,v)\leq s\}$. Since the great circle $S$ will be clear from the context, we omit reference to $S$ in the notation for $u^s$. Also note that $u^s=S$, whenever $s\geq\pi$.

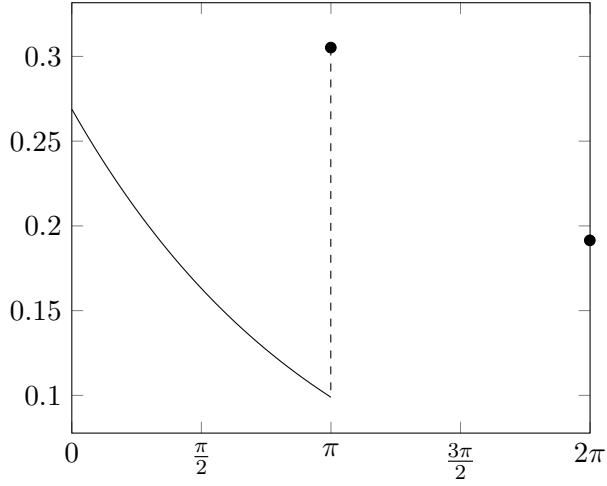
\begin{figure}[t]
\begin{center}
  \begin{tikzpicture}
    \begin{axis}[
     clip=false,
     xmin=0,xmax=6.2830,
     xtick={0,1.5707,3.1415,4.7122,6.2830},
     xticklabels={$0$, $\pi\over 2$, $\pi$, $3\pi\over 2$, $2\pi$}
     ]
      \addplot[domain=0:3.1415,samples=200,black]{exp(-x/3.1415)*0.268856}; 
      \addplot[domain=6.2828:6.2830,samples=2,black]{0.15526}; 
      \addplot[domain=6.2828:6.2830,samples=2,black]{0.15526*2}; 
    \end{axis}
\filldraw (6.85,5.1/2) circle (2pt);
\filldraw (6.85/2,5.1) circle (2pt);
\draw[dashed] (6.85/2,5.1) -- (6.85/2,0.48);
  \end{tikzpicture}
 \end{center}
\caption{The density of $\overline{\sL}_t$ on $[0,\pi)$ in the isotropic case together with the two atoms at $\pi$ and $2\pi$ for $d=2$ and $t=1$.}
\label{fig:DensityPoisson}
\end{figure}

\begin{theorem}\label{thm:PoissonEdgeLength}
Let $\overline{\sL}_t$ be the length of the typical edge in a Poisson  hypersphere tessellation on $\SSd$ of intensity $t>0$ and regular directional distribution $\kappa$. Then, for $0\leq s\leq 2\pi$,
\begin{align*}
\bP(\overline{\sL}_t\leq s) &= {e^{-t}\over 2t+e^{-t}}\,{\bf 1}\{s=2\pi\}+ {2t e^{-t} \over 2t+e^{-t}}\,{\bf 1}\{s\geq\pi\}\\
&\hspace{3cm}+ {2t\over 2t+e^{-t}}\int_{\SS_1}\int_S 1-e^{-t\kappa_S(v^s)}\,\kappa_S(\dint v)\kappa_\cap(\dint S).
\end{align*}
In particular, if $\kappa$ is the uniform distribution on $\SS_{d-1}$ (or, equivalently, if $\kappa_o$ is the normalized spherical Lebesgue measure on $\SSd$), then
\begin{align*}
\bP(\overline{\sL}_t\leq s) = {e^{-t}\over 2t+e^{-t}}\,{\bf 1}\{s=2\pi\}+ {2te^{-t} \over 2t+e^{-t}}\,{\bf 1}\{s\geq\pi\}+ {2t\over 2t+e^{-t}}\big(1-e^{-{(s\wedge\pi)t\over\pi}}\big),
\end{align*}
which is independent of $d$, and the random variable $\overline{\sL}_t$ has the  density 
$$
s\mapsto \overline\varrho_t(s):=\frac{2t^2}{\pi}\frac{e^{-\frac{st}{\pi}}}{2t+e^{-t}}
$$
with respect to the Lebesgue measure on $(0,\pi)$ (see also Figure \ref{fig:DensityPoisson}). 
\end{theorem}
\begin{proof}
Recall that for a spherical segment $e\in\PP^1$ we write $\cH^1(e)$ for its length (one-dimensional Hausdorff measure). Also recall the definition of the measures $\overline{\cF}_{X_t}=\overline{\cF}_{X_t}^{(1)}$ and $\overline{\FF}_{X_t}=\overline{\FF}_{X_t}^{(1)}$. Then, by the definition of the typical edge and the rotation invariance of $\cH^1$,
\begin{align*}
\bE|\cF_1(X_t)|\,\bP(\overline{\sL}_t\in\,\cdot\,) &= \int {\bf 1}\{\cH^1(e)\in\,\cdot\,\}\,\overline{\FF}_{X_t}(\dint e) \\
&= \bE \int {\bf 1}\{\cH^1(e)\in\,\cdot\,\}\,\overline{\cF}_{X_t}(\dint e)\\
&=\sum_{k=0}^\infty \bE\Big[{\bf 1}\{|\eta_t|=k\}\,\sum_{e\in\cF_1(X_t)}{\bf 1}\{\cH^1(e)\in\,\cdot\,\}\Big].
\end{align*}
We first notice that if $k\in\{0,\ldots,d-2\}$ and $|\eta_t|=k$, then $\cF_1(X_t)=\varnothing$. Also, if $|\eta_t|=d-1$, then $\cF_1(X_t)$ has precisely one element with length $2\pi$, and if $|\eta_t|=d$, then $\cF_1(X_t)$ consists of precisely $2d$ spherical segments with length $\pi$. Thus, using that $\bP(|\eta_t|=d-1)=e^{-t}{t^{d-1}\over(d-1)!}$ and $\bP(|\eta_t|=d)=e^{-t}{t^{d}\over d!}$ we have
\begin{equation}\label{eq:ProofEdgeLengthPoissonStart}
\begin{split}
\bE|\cF_1(X_t)|\,\bP(\overline{\sL}_t\in\,\cdot\,) &= e^{-t}{t^{d-1}\over(d-1)!}\,\delta_{2\pi}(\,\cdot\,)+2de^{-t}{t^d\over d!}\delta_\pi(\,\cdot\,)\\
& \qquad + \sum_{k=d+1}^\infty \bE\Big[{\bf 1}\{|\eta_t|=k\}\,\sum_{e\in\cF_1(X_t)}{\bf 1}\{\cH^1(e)\in\,\cdot\,\}\Big].
\end{split}
\end{equation}
In the remaining part of the proof we compute the sum on the right-hand side, which we abbreviate by $T(\,\cdot\,)$.
For \(I=\{i_1,\ldots,i_{d-1}\}\subset\{1,\ldots,k\}\), put
\[
S_I
:=
u_{i_1}^{\perp}\cap\cdots\cap u_{i_{d-1}}^{\perp}\cap\SSd.
\]
For \(\kappa_o^k\)-almost every tuple \((u_1,\ldots,u_k)\), the vectors $u_{i_1},\ldots,u_{i_{d-1}}$ are linearly independent and every edge is contained in exactly one of the
great circles \(S_I\). We write $X(u_1,\ldots,u_k)$ for the  hypersphere tessellation generated by the intersection of $\SSd$ with the $k$ hyperplanes $u_1^\perp,\ldots,u_k^\perp\subset\RRd1$. Moreover, we denote by $X(S,u_d,\ldots,u_k)$ the tessellation of $S\in\SS_1$ induced  by $u_d^\perp\cap S,\ldots,u_k^\perp\cap S$. Since $\kappa_0$ is regular and in particular diffuse, these intersections are pairs of distinct points on $S$ (almost surely). Similarly, we write $X(S;v_d,\ldots,v_k)$ for the tessellation of $S$ induced by the points $\pm v_d,\ldots,\pm v_k\in S$. 
Hence,
\begin{align*}
	T(\,\cdot\,)
	&=
	\sum_{k=d+1}^{\infty}e^{-t}{t^k\over k!}
	\int_{(\SSd)^k}
	\sum_{\substack{I\subset\{1,\ldots,k\}\\ |I|=d-1}}
	\sum_{\substack{e\in\cF_1(X(u_1,\ldots,u_k))\\ e\subset S_I}}
	{\bf 1}\{\cH^1(e)\in\,\cdot\,\}\,
	\kappa_o^k(\dint(u_1,\ldots,u_k))
	\\
	&=
	\sum_{k=d+1}^{\infty}e^{-t}{t^k\over k!}{k\choose d-1}
	\int_{(\SSd)^{d-1}}\int_{(\SSd)^{k-d+1}}
	\sum_{\substack{e\in\cF_1(X(u_1,\ldots,u_k))\\
			e\subset S_{\{1,\ldots,d-1\}}}}
	{\bf 1}\{\cH^1(e)\in\,\cdot\,\}
	\\
	&\hspace{3cm}\times
	\kappa_o^{k-d+1}(\dint(u_d,\ldots,u_k))
	\kappa_o^{d-1}(\dint(u_1,\ldots,u_{d-1}))
	\\
	&=
	\sum_{k=d+1}^{\infty}e^{-t}{t^k\over k!}{k\choose d-1}
	\int_{\SS_1}\int_{(\SSd)^{k-d+1}}
	\sum_{e\in\cF_1(X(S,u_d,\ldots,u_k))}
	{\bf 1}\{\cH^1(e)\in\,\cdot\,\}
	\\
	&\hspace{3cm}\times
	\kappa_o^{k-d+1}(\dint(u_d,\ldots,u_k))
	\kappa_\cap(\dint S)
	\\
	&=
	\sum_{k=d+1}^{\infty}e^{-t}{t^k\over k!}{k\choose d-1}
	\int_{\SS_1}\int_{S^{k-d+1}}
	\sum_{e\in\cF_1(X(S;v_d,\ldots,v_k))}
	{\bf 1}\{\cH^1(e)\in\,\cdot\,\}
	\\
	&\hspace{3cm}\times
	\kappa_S^{k-d+1}(\dint(v_d,\ldots,v_k))
	\kappa_\cap(\dint S),
\end{align*}
where we applied the definition of the measure $\kappa_\cap$ and afterwards that of $\kappa_S$. 
The points $v_d,\ldots,v_k$ partition the great circle $S$ into precisely $2(k-d+1)$ spherical segments, which come in antipodal pairs and are denoted by
$$
I_1^{(1)}(v_d,\ldots,v_k),I_1^{(2)}(v_d,\ldots,v_k),\ldots,I_{k-d+1}^{(1)}(v_d,\ldots,v_k),I_{k-d+1}^{(2)}(v_d,\ldots,v_k),
$$
see Figure \ref{fig:IntersectionPoints}. Thus,
\begin{align*}
T(\,\cdot\,) &=\sum_{k=d+1}^\infty e^{-t}{t^k\over k!}\,{k\choose d-1}\int_{\SS_1}\int_{S^{k-d+1}}\sum_{i=1}^{k-d+1}\Big[{\bf 1}\{\cH^1(I_i^{(1)}(v_d,\ldots,v_k))\in\,\cdot\,\}\\
&\hspace{4cm}+{\bf 1}\{\cH^1(I_i^{(2)}(v_d,\ldots,v_k))\in\,\cdot\,\}\Big]\,\kappa_S^{k-d+1}(\dint(v_d,\ldots,v_k))\kappa_\cap(\dint S).
\end{align*}
Since the lengths of the antipodal pairs of spherical segments $I_i^{(j)}(v_d,\ldots,v_k)$, $i\in\{1,\ldots,k-d+1\}$, $j\in\{1,2\}$, are identically distributed, we have that
\begin{align*}
T(\,\cdot\,) &=\sum_{k=d+1}^\infty e^{-t}{t^k\over k!}\,{k\choose d-1}(k-d+1)\int_{\SS_1}I(S,\cdot)\,\kappa_\cap(\dint S)
\end{align*}
with
\begin{align*}
I(S,\cdot) &:=\int_{S^{k-d+1}}\Big[{\bf 1}\{\cH^1(I_1^{(1)}(v_d,\ldots,v_k))\in\,\cdot\,\}\\
&\qquad\qquad\;+{\bf 1}\{\cH^1(I_1^{(2)}(v_d,\ldots,v_k))\in\,\cdot\,\}\Big]\,\kappa_S^{k-d+1}(\dint(v_d,\ldots,v_k)).
\end{align*}
To determine the integral $I(S,[0,s])$, for $s\in [0,2\pi]$, we use Fubini's theorem and write
\begin{align*}
I(S,[0,s])&=\int_S\int_{S^{k-d}}\Big[{\bf 1}\{\cH^1(I_1^{(1)}(v_d,\ldots,v_k))\leq s\}+{\bf 1}\{\cH^1(I_1^{(2)}(v_d,\ldots,v_k))\leq s\}\Big]\\
&\hspace{4cm}\times\kappa_S^{k-d}(\dint(v_{d+1},\ldots,v_k))\kappa_S(\dint v_d).
\end{align*}
The sum of the two indicator functions is $2$ if $v_{d+1},\ldots,v_k$ are not all outside of $v_d^s$ and $0$ otherwise. This means that
\begin{align*}
I(S) = 2\big(1-(1-\kappa_S(v_d^s))^{k-d}\big),
\end{align*}
and hence
\begin{align*}
T([0,s]) &= 2\sum_{k=d+1}^\infty e^{-t}{t^k\over k!}\,{k\choose d-1}(k-d+1)\int_{\SS_1}\int_S\big(1-(1-\kappa_S(v^s))^{k-d}\big)\,\kappa_S(\dint v)\kappa_\cap(\dint S)\\
&=\int_{\SS_1}\int_S 2\sum_{k=d+1}^\infty e^{-t}{t^k\over k!}\,{k\choose d-1}(k-d+1)\big(1-(1-\kappa_S(v^s))^{k-d}\big)\,\kappa_S(\dint v)\kappa_\cap(\dint S),
\end{align*}
where we put $v:=v_d$ and used the monotone convergence theorem. The remaining series can be evaluated explicitly:
\begin{align*}
&\sum_{k=d+1}^\infty e^{-t}{t^k\over k!}\,{k\choose d-1}(k-d+1)\big(1-(1-\kappa_S(v^s))^{k-d}\big)\\
&=e^{-t}\sum_{k=d+1}^\infty{t^k\over(d-1)!(k-d)!}\big(1-(1-\kappa_S(v^s))^{k-d}\big)\\
&={t^de^{-t}\over(d-1)!}\sum_{j=1}^\infty{t^{j}\over j!}\big(1-(1-\kappa_S(v^s))^{j}\big)\\
&={t^d\over(d-1)!}\big(1-e^{-t\kappa_S(v^s)}\big),
\end{align*}
which eventually leads to
\begin{align*}
T([0,s]) &= {2t^d\over(d-1)!}\int_{\SS_1}\int_S 1- e^{-t\kappa_S(v^s)}\,\kappa_S(\dint v)\kappa_\cap(\dint S).
\end{align*}
\begin{figure}[t]
\begin{center}
\begin{tikzpicture}
\draw[thick] (0,0) circle (3cm);
\filldraw (0,0) circle (2pt);
\filldraw (-0.5,2.958) circle (2pt); 
\filldraw (0.5,-2.958) circle (2pt); 
\draw[dashed] (-0.5,2.958) --  (0.5,-2.958);
\filldraw (-1.8,2.4) circle (2pt); 
\filldraw (1.8,-2.4) circle (2pt); 
\draw[dashed] (-1.8,2.4) --  (1.8,-2.4);
\filldraw (-2.9,0.768) circle (2pt); 
\filldraw (2.9,-0.768) circle (2pt); 
\draw[dashed] (-2.9,0.768) --  (2.9,-0.768);

\draw [red,thick,domain=102:125,->] plot ({2.8*cos(\x)}, {2.8*sin(\x)});
\node at (-0.9,2.2) {$I_1^{(1)}$};
\draw [red,thick,domain=102+180:125+180,->] plot ({2.8*cos(\x)}, {2.8*sin(\x)});
\node at (0.9,-2.2) {$I_1^{(2)}$};

\draw [red,thick,domain=130:163,->] plot ({2.8*cos(\x)}, {2.8*sin(\x)});
\node at (-1.9,1.4) {$I_2^{(1)}$};
\draw [red,thick,domain=130+180:163+180,->] plot ({2.8*cos(\x)}, {2.8*sin(\x)});
\node at (1.9,-1.4) {$I_2^{(2)}$};

\draw [red,thick,domain=170:190,->,dashed] plot ({2.8*cos(\x)}, {2.8*sin(\x)});
\draw [red,thick,domain=170+180:190+180,->,dashed] plot ({2.8*cos(\x)}, {2.8*sin(\x)});

\node at (2.7,2) {$S$};
\end{tikzpicture}
\end{center}
\caption{The tessellation within $S$.}
\label{fig:IntersectionPoints}
\end{figure}
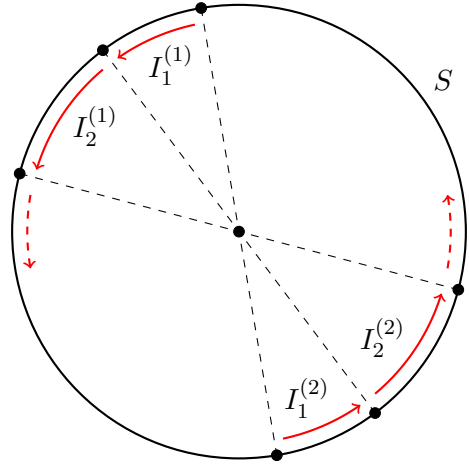

Plugging this back into \eqref{eq:ProofEdgeLengthPoissonStart} and using that
$$
\bE|\cF_1(X_t)|={t^{d-1}\over(d-1)!}(2t+e^{-t})
$$
from \cite[Equation (7.7)]{HugThaele18}, the proof of the first claim is complete.

To deduce the formula in the case where $\kappa_o$ is the uniform distribution on $\SSd$ we observe that in this case $\kappa_\cap$ is the uniform distribution on $\SS_1$ and $\kappa_S$ is the normalized Hausdorff measure $(2\pi)^{-1}\cH^1$ on $S$. Thus,
$$
\kappa_S(v^s)={1\over 2\pi}\cH^1(v^s)= {2(s\wedge\pi)\over 2\pi} = {s\wedge \pi\over\pi},
$$
independently of $v\in S$ and $S\in\SS_1$. Consequently,
$$
\int_{\SS_1}\int_S 1- e^{-t\kappa_S(v^s)}\,\kappa_S(\dint v)\kappa_\cap(\dint S) = 1-e^{-{t(s\wedge\pi)\over\pi}}
$$
and the proof is complete.
\end{proof}

Having determined the distribution of $\overline{\sL}_t$ we can now compute its moments, where from now on we concentrate on the isotropic case. We recall that $\gamma(a,x)$, $a,x>0$, stands
for the lower incomplete Gamma function.

\begin{corollary}\label{cor:4.2new}
Let $\overline{\sL}_t$ be the length of the typical edge in an isotropic Poisson  hypersphere tessellation on $\SSd$ of intensity $t>0$. If $m>0$, then
$$
\bE\overline{\sL}_t^m = (2\pi)^m\,{e^{-t}\over 2t+e^{-t}}+2\pi^m\,{te^{-t}\over 2t+e^{-t}}+{2\over\pi}{t^2\over 2t+e^{-t}}\,\Big({\pi\over t}\Big)^{m+1}\,\gamma(m+1,t).
$$
Especially,
$$
\bE\overline{\sL}_t = {2\pi\over 2t+e^{-t}}\qquad\text{and}\qquad \bE\overline{\sL}_t^2={4\pi^2(1-e^{-t})\over t(2t+e^{-t})}.
$$
\end{corollary}
\begin{proof}
We have that
\begin{align*}
\bE\overline{\sL}_t^m = (2\pi)^m\,{e^{-t}\over 2t+e^{-t}}+2\pi^m\,{te^{-t}\over 2t+e^{-t}}+{2\over\pi}{t^2\over 2t+e^{-t}} I(m,t)
\end{align*}
with $I(m,t)$ given by
$$
I(m,t) = \int_0^\pi s^me^{-{st\over\pi}}\,\dint s.
$$
Substituting $y={st\over\pi}$ in the definition of the lower incomplete Gamma function shows that
$$
I(m,t) = \Big({\pi\over t}\Big)^{m+1}\int_0^t y^me^{-y}\dint y = \Big({\pi\over t}\Big)^{m+1}\gamma(m+1,t).
$$
This proves the formula for $\bE\overline{\sL}_t^m$. In the special cases $m=1$ and $m=2$ we have that  $\gamma(1,t)=1-e^{-t}$, $\gamma(2,t)=1-(1+t)e^{-t}$ and $\gamma(3,t)=2-(t^2+2t+2)e^{-t}$, which, after simplification, yields the formulas for the first and second moment.
\end{proof}

Next, we record the high-intensity behavior of the typical edge length in the
isotropic Poisson hypersphere tessellation.  After the natural rescaling by
\(t/\pi\), the spherical length distribution becomes asymptotically
exponential.  Thus, at high intensity, the local length scale is the same as in
the corresponding flat  model. In what follows we indicate by $\xrightarrow[t\to\infty]{\rm TV}$ convergence as $t\to\infty$ in total variation.

\begin{corollary}\label{cor:PHTtToInfinity}
Let \(\overline{\sL}_t\) be the length of the typical edge of an
isotropic Poisson hypersphere tessellation on \(\mathbb S^d\). Then
\[
 {t\over \pi}\overline{\sL}_t
 \xrightarrow[t\to\infty]{\rm TV} Z,
\]
where \(Z\) has the exponential distribution with parameter \(1\).
\end{corollary}

\begin{proof}
We use the explicit distribution of \(\overline{\sL}_t\) from
Theorem~\ref{thm:PoissonEdgeLength}.  The atoms of
\(t\overline{\sL}_t/\pi\) are located at \(t\) and \(2t\), with masses
\(2te^{-t}/(2t+e^{-t})\) and \(e^{-t}/(2t+e^{-t})\), respectively.  Both masses
tend to zero as \(t\to\infty\). The density of the absolutely continuous part of
\(t\overline{\sL}_t/\pi\) is
\[
 p_t(x)={2t\over 2t+e^{-t}}e^{-x},\qquad 0<x<t.
\]
From  
\(\int_0^\infty |p_t(x)-e^{-x}|\,\dint x= (1+2te^t)^{-1} +\int_0^\infty \mathbf{1}\{t\le x\}e^{-x}\, \dint x \to0\) as $t\to\infty$ and the vanishing
of the two atoms, the convergence in total variation follows.
\end{proof}

\section{Maximal segment length distribution in spherical splitting tessellations}\label{sec:Splitting}


\pgfmathdeclarefunction{Iquad}{2}{%
  \pgfmathparse{%
    0.0506142681451881*pow(0.0198550717512319,#2)*exp(-#1*0.0198550717512319)
   +0.1111905172266870*pow(0.1016667612931870,#2)*exp(-#1*0.1016667612931870)
   +0.1568533229389440*pow(0.2372337950418360,#2)*exp(-#1*0.2372337950418360)
   +0.1813418916891810*pow(0.4082826787521750,#2)*exp(-#1*0.4082826787521750)
   +0.1813418916891810*pow(0.5917173212478250,#2)*exp(-#1*0.5917173212478250)
   +0.1568533229389440*pow(0.7627662049581640,#2)*exp(-#1*0.7627662049581640)
   +0.1111905172266870*pow(0.8983332387068130,#2)*exp(-#1*0.8983332387068130)
   +0.0506142681451881*pow(0.9801449282487680,#2)*exp(-#1*0.9801449282487680)%
  }%
}

\pgfmathdeclarefunction{rhoSTIT}{3}{%
  \pgfmathparse{%
    2*(#2)^2/pi
    *
    (#3*Iquad((#2)*(#1)/pi,#3))
    /
    (2*(#2)+#3*Iquad(#2,#3-2))%
  }%
}

\pgfmathdeclarefunction{rhoPoisson}{2}{%
  \pgfmathparse{%
    2*(#2)^2/(pi*(2*(#2)+exp(-#2)))
    *exp(-(#2)*(#1)/pi)%
  }%
}

\begin{figure}[t]
\begin{center}
\begin{tikzpicture}
\begin{axis}[
  width=0.72\columnwidth,
  height=6.2cm,
  domain=0:pi,
  samples=180,
  xmin=0,
  xmax=pi,
  ymin=0,
  ymax=1,
  xtick={0,pi/2,pi},
  xticklabels={$0$,$\pi/2$,$\pi$},
  legend style={
    draw=none,
    fill=none,
    at={(0.97,0.97)},
    anchor=north east
  }
]

\addplot[blue,thick] {rhoSTIT(x,3,2)};
\addlegendentry{$d=2$}

\addplot[orange,thick] {rhoSTIT(x,3,3)};
\addlegendentry{$d=3$}

\addplot[green!60!black,thick] {rhoSTIT(x,3,4)};
\addlegendentry{$d=4$}

\addplot[red,thick] {rhoSTIT(x,3,5)};
\addlegendentry{$d=5$}

\addplot[black,dashed,thick] {rhoPoisson(x,3)};
\addlegendentry{Poisson}

\end{axis}
\end{tikzpicture}
\end{center}

\caption{Density part of the distribution of \(\sL_t\) on \((0,\pi)\)
in the isotropic case for \(t=3\) and \(d=2,3,4,5\), together with the
density of \(\overline{\sL}_t\) in the Poisson hypersphere tessellation
(dashed). The curves illustrate the change in the dimension dependence
described in Section~\ref{sec:HighDimensions}.}
\label{fig:STITvsPoisson}
\end{figure}
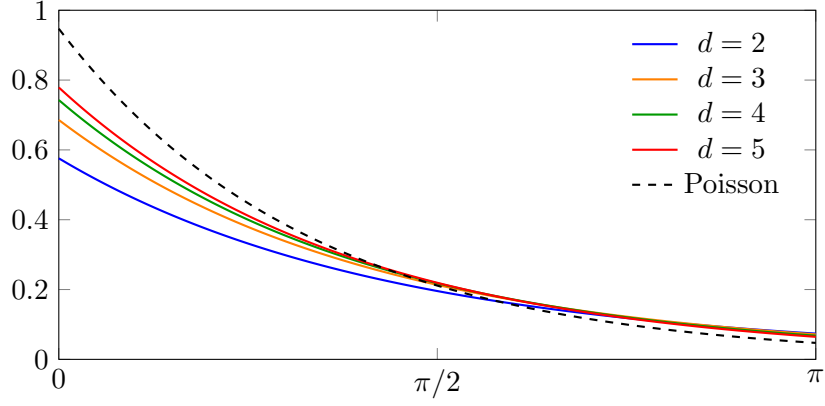

In \cite[Corollary 7.9]{HugThaele18} we established a link between the distribution of the typical maximal segment of a spherical splitting tessellation and the distribution of the typical edge in a Poisson  hypersphere tessellation. We will use this connection now to determine the distribution function of the length of the typical maximal segment of a splitting tessellation $Y_t$ on $\SSd$ with time parameter $t>0$ and regular directional distribution $\kappa$. Recall the definition of the measures $\kappa_\cap$ and $\kappa_S$, $S\in\SS_1$, from the discussion preceding Theorem \ref{thm:PoissonEdgeLength}. Again, $\gamma(a,x)$ will denote the lower incomplete gamma function.

\begin{theorem}\label{thm:LengthSTIT}
Let $\sL_t$ be the length of the typical maximal segment in a splitting tessellation on $\SSd$ with time parameter $t>0$ and regular directional distribution $\kappa$. Then
\begin{align*}
\bP(\sL_t\leq s) = {d\over 2t^d+d\gamma(d-1,t)}&\bigg[\gamma(d-1,t){\bf 1}\{s=2\pi\}+2\gamma(d,t){\bf 1}\{s\geq\pi\}\\
&\qquad\qquad+{2t^d\over d}-2\int_{\SS_1}\int_S{\gamma(d,t\kappa_S(v^s))\over\kappa_S(v^s)^d}\,\kappa_S(\dint v)\kappa_\cap(\dint S)\bigg],
\end{align*}
where for $s=0$ the integrand is understood as $t^d/d$ by continuous extension, see Figure \ref{fig:STITvsPoisson}.
\end{theorem}
\begin{proof}
We use \cite[Corollary 7.9]{HugThaele18}, Theorem \ref{thm:PoissonEdgeLength} and Fubini's theorem to see that
\begin{align*}
\bP(\sL_t\leq s) &= {d\over 2t^d+d\gamma(d-1,t)}\int_0^tu^{d-2}(2u+e^{-u})\bP(\overline{\sL}_u\leq s)\,\dint u\\
&={d\over 2t^d+d\gamma(d-1,t)}\bigg[{\bf 1}\{s=2\pi\}\int_0^t u^{d-2}e^{-u}\,\dint u+2\cdot{\bf 1}\{s\geq\pi\}\int_0^t u^{d-1}e^{-u}\,\dint u\\
&\qquad\qquad\qquad\qquad\qquad+2\int_{\SS_1}\int_S\int_0^tu^{d-1}(1-e^{-u\kappa_S(v^s)})\,\dint u\kappa_S(\dint v)\kappa_\cap(\dint S)\bigg]\\
&={d\over 2t^d+d\gamma(d-1,t)}\bigg[\gamma(d-1,t){\bf 1}\{s=2\pi\}+2\gamma(d,t){\bf 1}\{s\geq\pi\}\\
&\qquad\qquad\qquad\qquad\qquad+2\int_{\SS_1}\int_S\int_0^tu^{d-1}(1-e^{-u\kappa_S(v^s)})\,\dint u\kappa_S(\dint v)\kappa_\cap(\dint S)\bigg].
\end{align*}
Substituting $y=u\kappa_S(v^s)$, the inner integral in the last expression can be determined as follows:
\begin{align*}
\int_0^tu^{d-1}(1-e^{-u\kappa_S(v^s)})\,\dint u = {t^d\over d}-{1\over\kappa_S(v^s)^d}\int_0^{t\kappa_S(v^s)}y^{d-1}e^{-y}\,\dint y = {t^d\over d}-{\gamma(d,t\kappa_S(v^s))\over\kappa_S(v^s)^d}.
\end{align*}
This proves the claim.
\end{proof}

Next, we specialize the formula to the case that $\kappa$ is the uniform distribution on $\SS_{d-1}$, where the result can be expressed in terms of the lower incomplete gamma function.

\begin{corollary}\label{cor:LengthIso}
	Let $\sL_t$ be the length of the typical maximal segment in an isotropic splitting tessellation on $\SS^d$ with time parameter $t>0$. Then
	$$
	\bP(\sL_t=2\pi) = {d\gamma(d-1,t)\over 2t^d+d\gamma(d-1,t)},\qquad \bP(\sL_t=\pi) = {2d\gamma(d,t)\over 2t^d+d\gamma(d-1,t)}
	$$
	and on $(0,\pi)$ the random variable $\sL_t$ has the density
	\begin{equation}\label{eq:densityrhotds}
	s\mapsto \varrho_{t,d}(s)= {2\pi^d\over 2t^d+d\gamma(d-1,t)}{\gamma\left(d+1,{st\over\pi}\right)\over s^{d+1}} =\frac{2t^2}{\pi}\frac{\displaystyle{d\int_0^1x^de^{-\frac{st}{\pi}x}\,\dint x}}{\displaystyle{2t+d\int_0^1x^{d-2}e^{-tx}\, \dint x}}
	\end{equation}
	with respect to the Lebesgue measure on $(0,\pi)$.
\end{corollary}
\begin{proof}
	The formulas for $\bP(\sL_t=2\pi)$ and $\bP(\sL_t=\pi)$ are evident from Theorem \ref{thm:LengthSTIT}, and to determine the density of $\sL_t$ on $(0,\pi)$ we observe that in the isotropic case $\kappa_\cap$ is the uniform distribution on $\SS_1$, $\kappa_S$ is the normalized Hausdorff measure on $S$ and $\kappa_S(v^s)=s/\pi$, independently of $S$ and $v$. This yields for $s\in(0,\pi)$,
	$$
	\bP(\sL_t\leq s) = {d\over 2t^d+d\gamma(d-1,t)}\Big({2t^d\over d}-2\Big({\pi\over s}\Big)^{d}\gamma\Big(d,{st\over\pi}\Big)\Big)
	$$
	and differentiation with respect to $s$ yields 
    \begin{equation}\label{eq:altrhotds}
    \varrho_{t,d}(s)= {d\over 2t^d+d\gamma(d-1,t)}{2\over s^{d+1}}\Big(d\pi^d\gamma\Big(d,{st\over\pi}\Big)-(st)^de^{-{st\over\pi}}\Big),
    \end{equation}
    partial integration shows that 
    $$
    d\pi^d\gamma\Big(d,{st\over\pi}\Big)-(st)^de^{-{st\over\pi}}=\pi^d\gamma\Big(d+1,{st\over\pi}\Big),
    $$
   which yields the formula for the density.
\end{proof}

We record the case $d=2$ of Corollary \ref{cor:LengthIso}  separately, because it will be needed in the next section.

\begin{corollary}\label{cor:Length2d}
Let $\sL_t$ be the length of the typical maximal segment in an isotropic splitting tessellation on $\SS^2$ with time parameter $t>0$. Then,
$$
\bP(\sL_t=2\pi) = {1-e^{-t}\over t^2+1-e^{-t}},\qquad \bP(\sL_t=\pi) = {2-2(1+t)e^{-t}\over t^2+1-e^{-t}},
$$
and on $(0,\pi)$ the random variable $\sL_t$ has the density
$$
s\mapsto {2\over t^2+1-e^{-t}}\Big({2\pi^2\over s^3}-\Big({t^2\over s}+{2\pi t\over s^2}+{2\pi^2\over s^3}\Big)e^{-{st\over\pi}}\Big)
$$
with respect to the Lebesgue measure.
\end{corollary}
\begin{proof}
This follows from Corollary \ref{cor:LengthIso} and   explicit formulas for $\gamma(1,x)$ and $\gamma(3,x)$, $x>0$,  or from \eqref{eq:altrhotds} and explicit formulas for $\gamma(1,x)$ and $\gamma(2,x)$. 
\end{proof}

While the expected length $\EE\sL_t$ of the typical maximal segment in a splitting tessellation on $\SSd$ with time parameter $t>0$ has already been determined in \cite[Corollary 7.10]{HugThaele18}, the distributional result in Theorem \ref{thm:LengthSTIT} allows us to compute higher moments as well. For simplicity, we demonstrate this for the first two moments only and denote for $x>0$ by $E_1(x):=\int_x^\infty{e^{-s}\over s}\,\dint s$ the exponential integral and by $C_{\operatorname{EM}}$ the Euler-Mascheroni constant. Instead of using Corollary \ref{cor:Length2d} in the proof of the following corollary, we use \cite[Corollary 7.9]{HugThaele18} and Corollary \ref{cor:4.2new}.

\begin{corollary}
	Let $\sL_t$ be the length of the typical maximal segment in an isotropic splitting tessellation on $\SSd$ with time parameter $t>0$ and dimension $d\geq 2$. Then
	\begin{align*}
		\EE\sL_t &= {2\pi dt^{d-1}\over (d-1)(2t^d+d\gamma(d-1,t))},\\
		\EE\sL_t^2 & = \begin{cases}
		{4\pi^2\over t^2+1-e^{-t}}\big(C_{\operatorname{EM}}+E_1(t)+\log t\big) &: d=2\\[3pt]
		{4\pi^2 d\over (d-2)(2t^d+d\gamma(d-1,t))}
		\Big(t^{d-2}(1-e^{-t})-\gamma(d-1,t)\Big) &: d\geq 3.
		\end{cases}
	\end{align*}
\end{corollary}
\begin{proof}
	Put
	\[
	D_t:=2t^d+d\gamma(d-1,t).
	\]
	From 
	\cite[Corollary 7.9]{HugThaele18} we obtain
	\begin{equation}\label{eq:MomentMixture}
		\EE\sL_t^m
		=
		{d\over D_t}\int_0^t
		u^{d-2}(2u+e^{-u})\,
		\EE\overline{\sL}_u^{\,m}\,\dint u,
	\end{equation}
	for \(m\in\{1,2\}\), where \(\overline{\sL}_u\) is the length of the typical edge of the
	isotropic Poisson hypersphere tessellation with intensity \(u\).

	For \(m=1\), by Corollary \ref{cor:4.2new} we have  
    $(2u+e^{-u})\EE\overline{\sL}_u=2\pi$. 
	Substitution into \eqref{eq:MomentMixture} yields
	\[
	\EE\sL_t
	=
	{2\pi d\over D_t}\int_0^t u^{d-2}\,\dint u
	=
	{2\pi dt^{d-1}\over(d-1)D_t}.
	\]
	
	For \(m=2\), again  by Corollary \ref{cor:4.2new} 
	we have
	\[
	(2u+e^{-u})\EE\overline{\sL}_u^{\,2}
	=
	{4\pi^2(1-e^{-u})\over u},
	\]
	and hence
	\begin{equation}\label{eq:SecondMomentCorrect}
		\EE\sL_t^2
		=
		{4\pi^2d\over D_t}
		\int_0^t u^{d-3}(1-e^{-u})\,\dint u.
	\end{equation}
	
	If \(d=2\), then \(D_t=2(t^2+1-e^{-t})\) and
	\[
	\int_0^t{1-e^{-u}\over u}\,\dint u
	=
	C_{\operatorname{EM}}+\log t+E_1(t),
	\]
	which gives
	\[
	\EE\sL_t^2
	=
	{4\pi^2\over t^2+1-e^{-t}}
	\big(C_{\operatorname{EM}}+\log t+E_1(t)\big).
	\]
	
	If \(d\geq3\), then
	\begin{align*}
		\int_0^t u^{d-3}(1-e^{-u})\,\dint u
		&=
		{t^{d-2}\over d-2}-\gamma(d-2,t)
		=
		{t^{d-2}(1-e^{-t})-\gamma(d-1,t)\over d-2}.
	\end{align*}
	Substitution into \eqref{eq:SecondMomentCorrect} gives
	\[
	\EE\sL_t^2
	=
	{4\pi^2d\over(d-2)D_t}
	\Big(t^{d-2}(1-e^{-t})-\gamma(d-1,t)\Big),
	\]
	which completes the proof.
\end{proof}

Finally, we turn to the large-time limit of the length distribution of the typical maximal segment. This provides the analogue for splitting tessellations of Corollary \ref{cor:PHTtToInfinity}. 

\begin{corollary}
Let \(\sL_t\) be the length of the typical maximal segment of an
isotropic spherical splitting tessellation on \(\mathbb S^d\). Then
\[
 {t\over \pi}\sL_t
 \xrightarrow[t\to\infty]{\rm TV} Z_d,
\]
where \(Z_d\) is a random variable with density
\[
 h_d(x)
 =\frac{d\gamma(d+1,x)}{x^{d+1}},
 \qquad x>0.
\]
\end{corollary}
\begin{proof}
We use the explicit distribution of \(\sL_t\) from
Theorem~\ref{thm:LengthSTIT}.  Put
\(D_t:=2t^d+d\gamma(d-1,t)\).  The atoms of \(t\sL_t/\pi\) are located
at \(t\) and \(2t\), with masses \(2d\gamma(d,t)/D_t\) and
\(d\gamma(d-1,t)/D_t\), respectively.  Both atoms
vanish as \(t\to\infty\). The density of the absolutely continuous part of \(t\sL_t/\pi\) is
\[
 p_t(x)
 =
 {2dt^d\over D_t}\,
 {\gamma(d+1,x)\over x^{d+1}},\qquad 0<x<t.
\]
Consequently, for every \(0<x<t\),
\[
 p_t(x)\to
 h_d(x):=
 d\,{\gamma(d+1,x)\over x^{d+1}}.
\]
By Fubini's theorem, 
\[
 \int_0^\infty h_d(x)\,\dint x
 =\int_0^\infty \int_0^x \frac{d}{x^{d+1}}u^de^{-u}\, \dint u\,\dint x
 =\int_0^\infty [-x^{-d}]_u^\infty \, u^d e^{-u}\, \dint u=\int_0^\infty e^{-u}\, \dint u=1.
\]
Let $x>0$ be fixed. Since
$$
|p_t(x)-h_d(x)|=\frac{d\gamma(d-1,t)}{D_t}h_d(x)+\mathbf{1}\{t\le x\}h_d(x)\to 0
$$
as $t\to\infty$ and $h_d$ is integrable, we get $\int_0^\infty |p_t(x)-h_d(x)|\, \dint x\to 0$ as $t\to\infty$.  Combining this with the
vanishing of the two atoms proves convergence in total variation.
\end{proof}

\section{High-dimensional behaviour}\label{sec:HighDimensions}

\begin{figure}[t]
\begin{center}
\begin{tikzpicture}
\begin{groupplot}[
  group style={group size=2 by 1, horizontal sep=1.6cm},
  width=0.44\columnwidth,
  height=6.2cm,
  domain=0:pi,
  samples=180,
  xmin=0,
  xmax=pi,
  ymin=0,
  xtick={0,pi/2,pi},
  xticklabels={$0$,$\pi/2$,$\pi$},
  legend style={
    draw=none,
    fill=none,
    at={(0.97,0.97)},
    anchor=north east
  }
]

\nextgroupplot[
  title={$t=\frac34$},
  ymax=0.28
]
\addplot[blue,thick] {rhoSTIT(x,0.75,2)};
\addlegendentry{$d=2$}
\addplot[orange,thick] {rhoSTIT(x,0.75,3)};
\addlegendentry{$d=3$}
\addplot[green!60!black,thick] {rhoSTIT(x,0.75,4)};
\addlegendentry{$d=4$}
\addplot[red,thick] {rhoSTIT(x,0.75,5)};
\addlegendentry{$d=5$}
\addplot[black,dashed,thick] {rhoPoisson(x,0.75)};
\addlegendentry{Poisson}

\nextgroupplot[
  title={$t=2$},
  ymax=0.72
]
\addplot[blue,thick] {rhoSTIT(x,2,2)};
\addplot[orange,thick] {rhoSTIT(x,2,3)};
\addplot[green!60!black,thick] {rhoSTIT(x,2,4)};
\addplot[red,thick] {rhoSTIT(x,2,5)};
\addplot[black,dashed,thick] {rhoPoisson(x,2)};

\end{groupplot}
\end{tikzpicture}
\end{center}
\caption{Density part of the distribution of \(\sL_t\) on \((0,\pi)\) in the
isotropic case for \(d=2,3,4,5\), together with the density of
\(\overline{\sL}_t\) in the Poisson hypersphere tessellation (dashed). For
\(t=\frac34\) (left), the splitting densities approach the Poisson density
from below throughout \((0,\pi)\). For \(t=2\) (right), the curves cross the Poisson density.}
\label{fig:STITDensityDimensionDependence}
\end{figure}
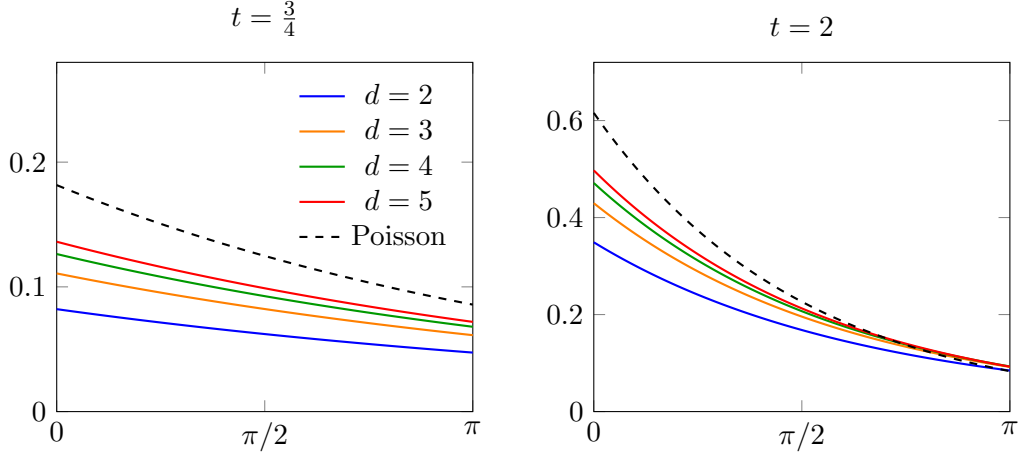

The explicit formulas obtained in the preceding two sections allow us to
compare the two tessellation models as the dimension tends to infinity.  We
first show that, for fixed \(t>0\), the length distribution of the
typical maximal segment in the spherical splitting tessellation converges to
the length distribution of the typical edge in the Poisson hypersphere
tessellation with the same intensity parameter as $d\to\infty$.  This convergence is natural in
view of the mixture representation from 
\cite[Corollary~7.9]{HugThaele18} we already used in the proof of Theorem \ref{thm:LengthSTIT}.  Namely, in terms of the length distributions considered here, this
representation says that the law of \(\sL_t\) is obtained by mixing the Poisson
edge-length laws with respect to the probability density
\[
	s\mapsto
	{d\,s^{d-2}(2s+e^{-s})\over 2t^d+d\gamma(d-1,t)},
	\qquad 0<s<t.
\]
As \(d\to\infty\), this density concentrates at the upper endpoint \(t\).
Thus the weak convergence to the Poisson edge-length distribution with
parameter \(t\) is not unexpected.  The purpose of the present section is to go
beyond this qualitative consequence of the mixture representation.  We derive
explicit error terms and look
more closely at the way in which the density and the two point masses approach
their Poisson counterparts. While the limiting distribution is the same, the dimension dependence of its individual components turns out to be more subtle.

For \(s\in[0,\pi]\), put \(a:=s/\pi\). From
\eqref{eq:densityrhotds} we obtain
\begin{equation}\label{eq:rho-normalized}
\varrho_{t,d}(\pi a)
=
\frac{2t^2}{\pi}
\frac{d\displaystyle\int_0^1u^de^{-atu}\,\dint u}
{2t+d\displaystyle\int_0^1u^{d-2}e^{-tu}\,\dint u}.
\end{equation}
The right-hand side extends continuously to \(a=0\) and \(a=1\). Recall that
the density of the typical edge length in the isotropic Poisson hypersphere
tessellation is
\[
\overline{\varrho}_t(\pi a)
=
\frac{2t^2}{\pi}\frac{e^{-at}}{2t+e^{-t}},
\qquad a\in[0,1].
\]

\begin{proposition}\label{prop:HighDimDistribution}
Fix \(t>0\). As \(d\to\infty\), the distribution of \(\sL_t\) converges in
total variation to the distribution of \(\overline{\sL}_t\). More precisely,
the convergence of the densities is uniform on \([0,\pi]\), and for every
\(d\geq2\),
\[
\sup_{s\in[0,\pi]}
\big|\varrho_{t,d}(s)-\overline{\varrho}_t(s)\big|
\leq
\frac{t}{\pi}
\left(
\frac1{d+1}+\frac{t}{d+2}
+
\frac{1}{2t+e^{-t}}
\left(\frac1{d-1}+\frac td\right)
\right).
\]
Moreover,
\[
\left|
\PP(\sL_t=2\pi)-\frac{e^{-t}}{2t+e^{-t}}
\right|
\leq
\frac{1}{2t+e^{-t}}
\left(\frac1{d-1}+\frac td\right)
\]
and
\[
\left|
\PP(\sL_t=\pi)-\frac{2te^{-t}}{2t+e^{-t}}
\right|
\leq
\frac{t}{d+1}
+
\frac{e^{-t}}{2t+e^{-t}}
\left(\frac1{d-1}+\frac td\right).
\]
In particular, for fixed \(t>0\), the uniform difference of the densities,
the differences of the two point masses, and the total variation distance
between the two length distributions are all of order \(O_t(d^{-1})\), where
the implicit constants may depend on \(t\), but not on \(d\).
\end{proposition}

\begin{proof}
Write \(a=s/\pi\in [0,1]\). We first estimate the two integrals appearing in
\eqref{eq:rho-normalized}. The substitution \(v=u^{d+1}\) gives
\[
A_{d,t}(a):=d\int_0^1u^de^{-atu}\,\dint u
=
\frac{d}{d+1}
\int_0^1e^{-atv^{1/(d+1)}}\,\dint v.
\]
Since \(0\leq a\leq1\) and \(|e^{-x}-e^{-y}|\leq|x-y|\) for \(x,y\geq0\),
we have
\[
\left|
e^{-atv^{1/(d+1)}}-e^{-at}
\right|
\leq
t\bigl(1-v^{1/(d+1)}\bigr),\qquad v\in [0,1].
\]
Consequently,
\begin{align*}
\left|
A_{d,t}(a)-e^{-at}
\right|
&\leq
\frac{1}{d+1}
+
\int_0^1
\left|e^{-atv^{1/(d+1)}}-e^{-at}\right|
\,\dint v\\
&\leq
\frac1{d+1}
+
t
\int_0^1
\bigl(1-v^{1/(d+1)}\bigr)\,\dint v 
=
\frac1{d+1}+\frac{t}{d+2}.
\end{align*}
This estimate is uniform in \(a\in[0,1]\).

For the denominator in \eqref{eq:rho-normalized}, the substitution
\(v=u^{d-1}\) yields
\[
B_{d,t}:=d\int_0^1u^{d-2}e^{-tu}\,\dint u
=
\frac{d}{d-1}
\int_0^1e^{-tv^{1/(d-1)}}\,\dint v.
\]
Arguing as before, we obtain
\begin{align*}
\left|
B_{d,t}-e^{-t}
\right|
&\leq \frac1{d-1}+
\int_0^1
\left|e^{-tv^{1/(d-1)}}-e^{-t}\right|
\,\dint v\\
&\leq \frac1{d-1}+
t\int_0^1
\bigl(1-v^{1/(d-1)}\bigr)\,\dint v
=
\frac1{d-1}+\frac td.
\end{align*}

We now compare the two densities. 
By \eqref{eq:rho-normalized},
\[
\varrho_{t,d}(\pi a)
=
\frac{2t^2}{\pi}\frac{A_{d,t}(a)}{2t+B_{d,t}},
\qquad
\overline{\varrho}_t(\pi a)
=
\frac{2t^2}{\pi}\frac{e^{-at}}{2t+e^{-t}}.
\]
It follows that
\begin{align*}
\big|\varrho_{t,d}(\pi a)-\overline{\varrho}_t(\pi a)\big|
&\leq
\frac{2t^2}{\pi}
\frac{|A_{d,t}(a)-e^{-at}|}{2t+B_{d,t}}+
\frac{2t^2}{\pi}
e^{-at}
\frac{|B_{d,t}-e^{-t}|}
{(2t+B_{d,t})(2t+e^{-t})}.
\end{align*}
Since \(B_{d,t}>0\) and \(e^{-at}\leq1\), the preceding estimates imply
\[
\big|\varrho_{t,d}(\pi a)-\overline{\varrho}_t(\pi a)\big|
\leq
\frac{t}{\pi}
\left(
\frac1{d+1}+\frac{t}{d+2}
+
\frac{1}{2t+e^{-t}}
\left(\frac1{d-1}+\frac td\right)
\right),
\]
uniformly in \(a\in[0,1]\). This proves the asserted uniform estimate for the
densities and, in particular, their uniform convergence.

We next consider the point masses. From Corollary~\ref{cor:LengthIso} and the 
substitution \(u=tx\), we get $\PP(\sL_t=2\pi)=\frac{B_{d,t}}{2t+B_{d,t}}$.
Hence
\begin{align*}
\left|
\PP(\sL_t=2\pi)-\frac{e^{-t}}{2t+e^{-t}}
\right|
&=
\frac{2t|B_{d,t}-e^{-t}|}
{(2t+B_{d,t})(2t+e^{-t})}\leq
\frac{|B_{d,t}-e^{-t}|}{2t+e^{-t}}\leq
\frac{1}{2t+e^{-t}}
\left(\frac1{d-1}+\frac td\right).
\end{align*}

The substitution \(v=u^d\) shows that
\[
C_{d,t}:=d\int_0^1u^{d-1}e^{-tu}\,\dint u
=
\int_0^1e^{-tv^{1/d}}\,\dint v,
\]
and therefore
\[
\left|
C_{d,t}-e^{-t}
\right|
\leq
t\int_0^1\bigl(1-v^{1/d}\bigr)\,\dint v
=
\frac{t}{d+1}.
\]
From Corollary~\ref{cor:LengthIso} and the 
substitution \(u=tx\), we get $\PP(\sL_t=\pi)=2t\frac{C_{d,t}}{2t+B_{d,t}}$, and hence
\begin{align*}
\left|
\PP(\sL_t=\pi)-\frac{2te^{-t}}{2t+e^{-t}}
\right|
&\leq
\frac{2t}{2t+B_{d,t}}
\left|
C_{d,t}-e^{-t}
\right|+
2te^{-t}
\frac{|B_{d,t}-e^{-t}|}
{(2t+B_{d,t})(2t+e^{-t})}\\
&\leq
\frac{t}{d+1}
+
\frac{e^{-t}}{2t+e^{-t}}
\left(\frac1{d-1}+\frac td\right).
\end{align*}
This proves the estimates for the two atoms.

The three estimates obtained above show directly that the total variation distance is of
order \(O_t(d^{-1})\). In particular, it tends to zero as \(d\to\infty\),
which completes the proof.
\end{proof}

We next study more closely how the density approaches its high-dimensional
limit. Figure~\ref{fig:STITDensityDimensionDependence} illustrates that the
direction of this approximation depends on the parameters: for some choices the
splitting densities increase towards the Poisson density, whereas for others
they cross it and eventually approach it from above. We now determine the
parameter ranges governing these different behaviours.

\begin{proposition}\label{prop:HighDimDensity}
Fix \(t>0\) and \(s\in[0,\pi]\).

\begin{enumerate}
\item[{\rm (i)}]
If \(st\leq\pi\), then
\(d\mapsto\varrho_{t,d}(s)\) is strictly increasing for \(d\geq2\), and
\[
\varrho_{t,d}(s)<\overline{\varrho}_t(s),
\qquad d\geq2.
\]

\item[{\rm (ii)}]
If
\[
\frac{st}{\pi}
>
1+\frac{(t+1)e^{-t}}{2t+e^{-t}},
\]
then \(d\mapsto\varrho_{t,d}(s)\) is strictly decreasing for all sufficiently
large \(d\) and approaches \(\overline{\varrho}_t(s)\) from above. If
\[
\frac{st}{\pi}
\leq
1+\frac{(t+1)e^{-t}}{2t+e^{-t}},
\]
then it is strictly increasing for all sufficiently large \(d\) and approaches
\(\overline{\varrho}_t(s)\) from below.
\end{enumerate}
\end{proposition}

\begin{proof}
We first prove part (i) and put again \(a=s/\pi\). For real \(x\geq2\) consider the expression
\[
\frac{x\displaystyle\int_0^1u^xe^{-atu}\,\dint u}
{2t+x\displaystyle\int_0^1u^{x-2}e^{-tu}\,\dint u}.
\]
By \eqref{eq:rho-normalized}, its value at \(x=d\), multiplied by
\(2t^2/\pi\), is \(\varrho_{t,d}(s)\). We show that the numerator is strictly increasing and the denominator is
strictly decreasing. In the numerator use \(u=e^{-y}\). If \(Z\) is
exponentially distributed with parameter one, then
\[
x\int_0^1u^xe^{-atu}\,\dint u
=
\EE\exp\left(
-\frac{Z}{x}-at\,e^{-Z/x}
\right).
\]
For \(at\leq1\), the function $y\longmapsto e^{-y-at e^{-y}}$
is strictly decreasing on \((0,\infty)\), since the derivative of its
logarithm is \(-1+at e^{-y}<0\). If \(x_2>x_1\), then
\(Z/x_2<Z/x_1\) almost surely. It follows that
\[
\exp\left(
-\frac{Z}{x_2}-at\,e^{-Z/x_2}
\right)
>
\exp\left(
-\frac{Z}{x_1}-at\,e^{-Z/x_1}
\right)
\]
almost surely, and hence the numerator is strictly increasing in \(x\).

For the integral occurring in the denominator, the same substitution gives
\[
B_{x,t}:=x\int_0^1u^{x-2}e^{-tu}\,\dint u
=
\EE\exp\left(
\frac{Z}{x}-t e^{-Z/x}
\right).
\]
The function \(y\mapsto e^{y-te^{-y}}\) is strictly increasing on
\((0,\infty)\). Thus the last expectation is strictly decreasing in \(x\).
The function $x\mapsto B_{x,t}$ is therefore strictly decreasing. Since all quantities
are positive, the quotient is strictly increasing. This proves the first
assertion in (i). Proposition~\ref{prop:HighDimDistribution} identifies its
limit with \(\overline{\varrho}_t(s)\), which gives the strict inequality.

We turn to part (ii). The required asymptotic expansion follows from repeated
integration by parts. If \(\phi\in C^4([0,1])\), that is, if \(\phi:[0,1]\to\RR\) is four times continuously differentiable, then, as \(x\to\infty\),
\begin{align}
\int_0^1u^x\phi(u)\,\dint u
&=
\frac{\phi(1)}{x}
-\frac{\phi(1)+\phi'(1)}{x^2}
+\frac{\phi(1)+3\phi'(1)+\phi''(1)}{x^3}
\notag\\
&\quad
-\frac{\phi(1)+7\phi'(1)+6\phi''(1)+\phi'''(1)}{x^4}
+O(x^{-5}).
\label{eq:HighDimEndpoint}
\end{align}
Indeed, four integrations by parts give the same expression with the
denominators \(x+1\), \((x+1)(x+2)\), and so on. The remaining integral
contains \(\phi^{(4)}\) and is \(O(x^{-5})\). Expanding the denominators in
powers of \(1/x\) yields \eqref{eq:HighDimEndpoint}.

Put \(c:=at\) and \(B:=2t+e^{-t}\). Applying
\eqref{eq:HighDimEndpoint} to \(\phi(u)=e^{-cu}\) gives
\begin{align*}
A_{d,t}(a)
=
e^{-c}\bigg(
1+\frac{c-1}{d}
+\frac{c^2-3c+1}{d^2}
+\frac{c^3-6c^2+7c-1}{d^3}
+O(d^{-4})
\bigg).
\end{align*}
Applying the same expansion with exponent \(d-2\), or equivalently expanding
the corresponding integration-by-parts formula directly, gives
\begin{align}\label{eq:expandBdt}
B_{d,t}
=
e^{-t}\bigg(
1+\frac{t+1}{d}
+\frac{t^2+t+1}{d^2}
+\frac{t^3+t+1}{d^3}
+O(d^{-4})
\bigg).
\end{align}
Consequently,
\begin{equation}\label{eq:HighDimDensityExpansion}
\frac{\varrho_{t,d}(s)}{\overline{\varrho}_t(s)}
=
1+\frac{Q}{d}+\frac{R}{d^2}+\frac{S}{d^3}+O(d^{-4}),
\end{equation}
where \(S\) is a constant whose precise value will not be needed,
\[
Q
=
c-1-\frac{(t+1)e^{-t}}{B},
\]
and
\[
R
=
c^2-3c+1
+\left(\frac{(t+1)e^{-t}}{B}\right)^2
-\frac{(t^2+t+1)e^{-t}}{B}
-(c-1)\frac{(t+1)e^{-t}}{B}.
\]

If \(Q>0\), \eqref{eq:HighDimDensityExpansion} shows that
\(\varrho_{t,d}(s)>\overline{\varrho}_t(s)\) for all sufficiently large
\(d\). If \(Q<0\), the reverse inequality holds. To determine the direction of
monotonicity, subtract the expansion at \(d\) from the one at \(d+1\). This
gives
\[
\frac{\varrho_{t,d+1}(s)-\varrho_{t,d}(s)}
{\overline{\varrho}_t(s)}
=
-\frac{Q}{d^2}+\frac{Q-2R}{d^3}+O(d^{-4}).
\]
Thus the sequence is eventually decreasing when \(Q>0\), and eventually
increasing when \(Q<0\).

It remains to consider \(Q=0\). Put $C:=\frac{(t+1)e^{-t}}{B}$.
Then \(c-1=C\), and the expression for \(R\) simplifies to
\[
R
=
C^2-C-1-\frac{(t^2+t+1)e^{-t}}{B}.
\]
Since \(0<C<1\), one has \(R<0\). Hence
\eqref{eq:HighDimDensityExpansion} shows that the density approaches the
Poisson density from below. Moreover, the preceding difference expansion
reduces to
\[
\frac{\varrho_{t,d+1}(s)-\varrho_{t,d}(s)}
{\overline{\varrho}_t(s)}
=
-\frac{2R}{d^3}+O(d^{-4}),
\]
which is positive for all sufficiently large \(d\). This proves the remaining
case.
\end{proof}

In particular, \(d\mapsto\varrho_{t,d}(0)\) is strictly increasing for every
\(t>0\). Moreover, if \(0<t\leq1\), then
\(\varrho_{t,d}(s)<\overline{\varrho}_t(s)\) for all
\(s\in[0,\pi]\) and \(d\geq2\). At \(s=\pi\), the transition in the eventual
monotonicity occurs at the unique \(t>1\) satisfying
\[
t(t-1)=e^{-t},
\]
which is approximately \(1.2353462335\).

We finally consider the two point masses. Their limiting values have already
been identified in Proposition~\ref{prop:HighDimDistribution}. Their behaviour
as functions of the dimension is, however, different.

\begin{proposition}\label{prop:HighDimAtoms}
Fix \(t>0\).

\begin{enumerate}
\item[{\rm (i)}]
The map \(d\mapsto\PP(\sL_t=2\pi)\) is strictly decreasing for \(d\geq2\)
and converges to \(e^{-t}/(2t+e^{-t})\).

\item[{\rm (ii)}]
If \(2t^2\leq e^{-t}\), then
\(d\mapsto\PP(\sL_t=\pi)\) is strictly increasing for all sufficiently large
\(d\) and approaches \(2te^{-t}/(2t+e^{-t})\) from below. If
\(2t^2>e^{-t}\), then it is strictly decreasing for all sufficiently large
\(d\) and approaches the same limit from above.
\end{enumerate}
\end{proposition}

\begin{proof}
For the atom at \(2\pi\), we have
\[
\PP(\sL_t=2\pi)
=
\frac{B_{d,t}}
{2t+B_{d,t}}.
\]
The function \(x\mapsto x/(2t+x)\) is strictly increasing, so it is enough to
recall that $d\mapsto B_{d,t}x$
is strictly decreasing in \(d\), as shown in the proof of Proposition \ref{prop:HighDimDensity} (i). This proves the strict monotonicity in (i),
and its limit has already been determined in
Proposition~\ref{prop:HighDimDistribution}.

For the atom at \(\pi\), write
\[
\PP(\sL_t=\pi)
=
2t
\frac{C_{d,t}}
{2t+B_{d,t}}.
\]
Using \eqref{eq:HighDimEndpoint}, we obtain
\[
C_{d,t}
=
e^{-t}\left(
1+\frac{t}{d}
+\frac{t^2-t}{d^2}
+\frac{t^3-3t^2+t}{d^3}
+O(d^{-4})
\right),
\]
whereas the expansion of $B_{d,t}$ is already given at \eqref{eq:expandBdt}. 
After division by the limiting Poisson mass, this gives
\begin{equation}\label{eq:HighDimAtomExpansion}
\frac{\PP(\sL_t=\pi)}
{2te^{-t}/(2t+e^{-t})}
=
1+
\frac{2t^2-e^{-t}}{(2t+e^{-t})d}
+\frac{R_t}{d^2}
+\frac{S_t}{d^3}
+O(d^{-4}),
\end{equation}
where \(R_t\) and \(S_t\) depend only on \(t\). If \(2t^2>e^{-t}\), the
first correction term is positive, so the mass is above its Poisson limit for
all sufficiently large \(d\). Subtracting
\eqref{eq:HighDimAtomExpansion} at consecutive dimensions shows that the
difference is negative for all sufficiently large \(d\). If
\(2t^2<e^{-t}\), the signs are reversed.

Suppose finally that \(2t^2=e^{-t}\). The coefficient of \(d^{-1}\) in
\eqref{eq:HighDimAtomExpansion} vanishes. Substituting
\(e^{-t}=2t^2\) into the coefficient of \(d^{-2}\) gives
\[
R_t=-\frac{t(t+2)}{t+1}<0.
\]
Thus the mass approaches its Poisson limit from below. Subtracting the
expansion at consecutive dimensions gives
\[
\frac{\PP(\sL_t=\pi)|_{d+1}-\PP(\sL_t=\pi)|_d}
{2te^{-t}/(2t+e^{-t})}
=
\frac{2t(t+2)}{(t+1)d^3}+O(d^{-4}),
\]
which is positive for all sufficiently large \(d\). This proves (ii).
\end{proof}

The equation \(2t^2=e^{-t}\) has a unique positive solution, approximately
\(0.5398352769\). Thus the point mass at \(2\pi\) always approaches its Poisson
counterpart monotonically from above, whereas the eventual direction of
approach of the point mass at \(\pi\) depends on \(t\). Together with
Proposition~\ref{prop:HighDimDensity}, this shows that the complete length
distribution has a simple Poisson limit in high dimensions, even though its
individual components can approach that limit in different ways.

\section{A Mecke-type formula for spherical maximal polytopes}\label{sec:Mecke}

Let $(M_t)_{t\geq 0}$ be the process of birth-time marked $(d-1)$-dimensional spherical maximal polytopes, where each $M_t$, $t\geq 0$, is based on the splitting tessellation process $(Y_u)_{0\leq u\leq t}$ up to time $t$. Using the description and the notation of Section \ref{sec:Prelim}, $M_t$ consists of those pairs $(c_i\cap S_i,s_i)$, where the triple $(s_i,c_i,S_i)$ is a jump time of the process $(Y_u)_{u\geq 0}$ with $s_i\leq t$. For $s>0$ we write $M_{t,+s}:=\{(m,\beta+s):(m,\beta)\in M_t\}$ for $M_t$, but with the birth times of the $(d-1)$-dimensional spherical maximal polytopes shifted by $s$. For a birth-time marked spherical maximal polytope $(p,\beta(p))\in M_t$,    
we write $M_t\cap p$ for the collection of sets obtained by intersecting $p$ with all maximal polytopes born after time $\beta(p)$ which intersect the relative interior of $p$, that is, $M_t\cap p=\{m\cap p:(m,r)\in M_t,r>\beta(p), m \cap {\rm relint}\,p\neq\varnothing\}$. Let $(\widetilde{M}_t)_{t\ge 0} $ be another process of birth-time marked $(d-1)$-dimensional spherical maximal polytopes, where each $\widetilde{M}_t$ is based on a splitting tessellation process $(\widetilde{Y}_u)_{0\leq u\leq t}$ up to time $t$.  If $c\in Y_s$ for $0<s<t$ and $S\in\SS_{d-1}[c]$, we shall write $\widetilde{M}_t\wedge(c\cap S^{(\pm)})$ for the restriction of $\widetilde{M}_t$ to $c\cap S^{(\pm)}$, where $S^{(\pm)}$ means $S^+$, $S^-$ or $S$. Formally, $\widetilde{M}_t\wedge(c\cap S^{(\pm)})=\{(m\cap(c\cap S^{(\pm)}),\beta):(m,\beta)\in \widetilde{M}_t,\beta>s\}$.

The next result is a spherical analogue of Proposition 3.2 in \cite{NNTW} and will be the key tool in Section \ref{sec:SphericalMaximalSegments}. Since its structure parallels the classical Mecke identity for Poisson point processes, we call it the Mecke-type formula for spherical maximal polytopes. We give two proofs. The first relies on the Poisson point process representation of the splitting tessellation process and adapts the argument from Theorem 3.1 in \cite{NNTW} to the spherical setting. The second is more direct and uses the explicit distributional description of the spherical splitting tessellation process. This latter approach is specific to the spherical framework considered here and has no global analogue in the Euclidean setting.

We write $\PP_{d-1}^d:=\{p\in\PP^d:\dim p=d-1\}$ and $\cF_{\rm fin}(\PP_{d-1}^d\times(0,\infty))$ for the space of finite subsets of $\PP_{d-1}^d\times(0,\infty)$.

\begin{theorem}\label{prop:Mecke}
If $t>0$ and  \(g:\PP_{d-1}^d\times(0,\infty)\times\cF_{\rm fin}(\PP_{d-1}^d\times(0,\infty))\to\RR\) is a non-negative measurable function, then 
\begin{align*}
&\EE \sum_{(p,s)\in M_t}g(p,s,M_t\cap p)\\
& = \EE\int_0^t\sum_{c\in Y_s}\int_{\SS_{d-1}[c]}g\big(c\cap S,s,((M_{t-s,+s}^{(1)}\wedge(c\cap S^+)\cup (M_{t-s,+s}^{(2)}\wedge(c\cap S^-))\wedge (c\cap S)\big)\,\kappa(\dint S)\,\dint s,
\end{align*}
where $(M_t^{(1)})_{t\geq 0}$ and $(M_t^{(2)})_{t\geq 0}$ are two independent random copies of $(M_t)_{t\geq 0}$.
\end{theorem}
\begin{proof}[Proof based on the description by a Poisson point process]
	We use the description of the process in terms of the Poisson point process
	\(\Sigma\). Let us write
	\[
	\mu:=\sigma_d\otimes\Leb_-\otimes\bP_\Pi
	\]
	for the intensity measure of \(\Sigma\). For
	\(\eta=(x,r,\psi)\in\SSd\times(-\infty,0)\times
	\sN(\SS_{d-1}\times(0,\infty))\),   a cell \(c\in\PP^d\) and  $\sigma\in \Sigma$, we write
	\[
	I(c,\sigma,\eta)
	:=
	{\bf 1}\Big\{
	x\in c,\
	r=\max\{r'\in(-\infty,0):(x',r',\psi')\in\sigma,\ x'\in c\}
	\Big\}.
	\]
	Thus, almost surely, \(I(c,\sigma,\eta)=1\) means that \(\eta\) is the triplet selected by the
	cell \(c\) in the realization \(\sigma\). Moreover, for
	\((S,s)\in\psi\), we write
	\[
	J(c,\psi,S,s)
	:=
	{\bf 1}\Big\{
	S\in\SS_{d-1}[c],\
	s=\min\{s'>\beta(c):(S',s')\in\psi,\ S'\in\SS_{d-1}[c]\}
	\Big\}.
	\]
	In this notation the splitting of a cell \(c\) by $S$ at time $s$ in the realization \(\sigma\)
	is encoded by \(I(c,\sigma,\eta)J(c,\psi,S,s)=1\).
	
	Let \(\bP_M\) be the distribution of the process \((M_t)_{t\geq0}\), and let
	\(\bP_\Sigma\) be the distribution of \(\Sigma\). Then
	\begin{align*}
		L
		&:= \EE \sum_{(p,s)\in M_t}g(p,s,M_t\wedge p)\\
		&=
		\int \sum_{(p,s)\in m_t}g(p,s,m_t\wedge p)\,\bP_M(\dint m)\\
		&=
		\int \sum_{(p,s)\in m_t(\sigma)}
		g(p,s,m_t(\sigma)\wedge p)\,\bP_\Sigma(\dint\sigma),
	\end{align*}
	where \(m_t(\sigma)\) denotes the realization of \(M_t\) generated by
	\(\sigma\). Writing \(y(\sigma)\) for the corresponding realization of the
	splitting tessellation process, this can be rewritten as
	\begin{align*}
		L
		&=
		\int
		\sum_{c\in y(\sigma)}
		\sum_{\eta=(x,r,\psi)\in\sigma}
		\sum_{ (S,s)\in\psi } {\bf 1}\{s\le t\}
		g(c\cap S,s,m_t(\sigma)\wedge (c\cap S))
		I(c,\sigma,\eta)J(c,\psi,S,s)\,
		\bP_\Sigma(\dint\sigma).
	\end{align*}
	Here and below the sum over \(c\in y(\sigma)\) is understood as the sum over
	all cells which are present during some non-empty time interval in the
	realization \(y(\sigma)\).
	
	We now apply the Mecke formula \cite[Theorem 4.1]{LP} for Poisson point processes to \(\Sigma\). This
	gives
	\begin{align*}
		L
		&=
		\int\int
		\sum_{c\in y(\sigma+\delta_\eta)}
		\sum_{ (S,s)\in\psi } {\bf 1}\{s\le t\}
		g(c\cap S,s,m_t(\sigma+\delta_\eta)\wedge(c\cap S))\\
		&\hspace{3.8cm}\times
		I(c,\sigma+\delta_\eta,\eta)J(c,\psi,S,s)\,
		\mu(\dint\eta)\,\bP_\Sigma(\dint\sigma).
	\end{align*}
	Writing again \(\eta=(x,r,\psi)\), and applying the Mecke formula to the
	Poisson point process \(\psi\) with intensity measure \(\kappa\otimes\Leb_+\), we obtain
	\begin{align*}
		L
		&=
		\int\int_{\SSd}\int_{-\infty}^0\int
		\int_{\SS_{d-1}}\int_0^t
		\sum_{c\in y(\sigma+\delta_{(x,r,\psi+\delta_{(S,s)})})}
		g\big(c\cap S,s,
m_t(\sigma+\delta_{(x,r,\psi+\delta_{(S,s)})})\wedge(c\cap S)\big)\\
		&\hspace{2.5cm}\times
		I\big(c,\sigma+\delta_{(x,r,\psi+\delta_{(S,s)})},
		(x,r,\psi+\delta_{(S,s)})\big)\\
		&\hspace{2.5cm}\times
		J\big(c,\psi+\delta_{(S,s)},S,s\big)\,
		\dint s\,\kappa(\dint S)\,\bP_\Pi(\dint\psi)\,\dint r\,
		\sigma_d(\dint x)\,\bP_\Sigma(\dint\sigma).
	\end{align*}
	The last indicator has a simple interpretation. Adding the point \((S,s)\) to
	\(\psi\) affects the construction only at time \(s\). Therefore, up to a
	null set of times, the product of the two indicators in the last
	display is equal to
	\[
	I(c,\sigma+\delta_{(x,r,\psi)},(x,r,\psi))\,
	{\bf 1}\{c\in y(\sigma+\delta_{(x,r,\psi)})_s\}\,
	{\bf 1}\{S\in\SS_{d-1}[c]\}.
	\]
	Moreover, on this event, the realization
	\(\sigma+\delta_{(x,r,\psi+\delta_{(S,s)})}\) coincides, up to time \(t\), with
	the realization obtained from \(\sigma+\delta_{(x,r,\psi)}\) by forcing the
	cell \(c\) to be split by \(S\) at time \(s\). We denote the latter realization
	by
	\[
	y(\sigma+\delta_{(x,r,\psi)},\oslash_{c,S,s}).
	\]
	Consequently,
	\begin{align*}
		L
		&=
		\int\int
		\int_{\SS_{d-1}}\int_0^t
		\sum_{c\in y(\sigma+\delta_\eta)_s}
		g\big(c\cap S,s,
	 m_t(y(\sigma+\delta_\eta,\oslash_{c,S,s}))\wedge(c\cap S)\big)\\
		&\hspace{3.8cm}\times
		I(c,\sigma+\delta_\eta,\eta)\,
		{\bf 1}\{S\in\SS_{d-1}[c]\}\,
		\dint s\,\kappa(\dint S)\,\mu(\dint\eta)\,\bP_\Sigma(\dint\sigma).
	\end{align*}
	
	We now use the Mecke formula for \(\Sigma\) once more, this time in its
	backward form. This yields
	\begin{align*}
		L
		&=
		\int
		\int_{\SS_{d-1}}\int_0^t
		\sum_{c\in y(\sigma)_s}
		\sum_{\eta\in\sigma}
		g\big(c\cap S,s,
		m_t(y(\sigma,\oslash_{c,S,s}))\wedge(c\cap S)\big)\\
		&\hspace{3.8cm}\times
		I(c,\sigma,\eta)\,
		{\bf 1}\{S\in\SS_{d-1}[c]\}\,
		\dint s\,\kappa(\dint S)\,\bP_\Sigma(\dint\sigma).
	\end{align*}
	For fixed \(\sigma\) and \(c\in y(\sigma)_s\), there is almost surely exactly
	one point \(\eta\in\sigma\) with \(I(c,\sigma,\eta)=1\), namely the triplet
	selected by \(c\). Hence the inner sum over \(\eta\in\sigma\) is equal to one.
	Thus
	\begin{align}
		L
		&=
		\int
		\int_0^t
		\sum_{c\in y(\sigma)_s}
		\int_{\SS_{d-1}[c]}
		g\big(c\cap S,s,
		m_t(y(\sigma,\oslash_{c,S,s}))\wedge(c\cap S)\big)\,
		\kappa(\dint S)\,\dint s\,\bP_\Sigma(\dint\sigma).
		\label{eq:MeckeForcedSplit}
	\end{align}
	
	It remains to identify the process after the forced split. Conditional on the
	past up to time \(s\), and after the cell \(c\) has been split by \(S\), the
	future evolutions in the two daughter cells \(c\cap S^+\) and \(c\cap S^-\)
	are independent. After shifting birth times by \(s\), they have the same law as
	two independent copies of the original maximal-polytope process restricted to
	the corresponding daughter cells. Hence
	\[
	m_t(y(\sigma,\oslash_{c,S,s}))\wedge(c\cap S)
	\stackrel{d}{=}
	\Big(
	(M_{t-s,+s}^{(1)}\wedge(c\cap S^+))
	\cup
	(M_{t-s,+s}^{(2)}\wedge(c\cap S^-))
	\Big)\wedge(c\cap S),
	\]
	where \((M_u^{(1)})_{u\geq0}\) and \((M_u^{(2)})_{u\geq0}\) are independent
	copies of \((M_u)_{u\geq0}\), independent of \(Y_s\). Substituting this
	identification into \eqref{eq:MeckeForcedSplit} and rewriting the integral with
	respect to \(\bP_\Sigma\) as an expectation over \(Y_s\), we arrive at
	\begin{align*}
		L
		&=
		\EE\int_0^t
		\sum_{c\in Y_s}
		\int_{\SS_{d-1}[c]}
		g\Big(c\cap S,s,\\
        &\hspace{3cm}
		\Big(
		(M_{t-s,+s}^{(1)}\wedge(c\cap S^+))
		\cup
		(M_{t-s,+s}^{(2)}\wedge(c\cap S^-))
		\Big)\wedge(c\cap S)\Big)
		\,\kappa(\dint S)\,\dint s .
	\end{align*}
	This is the asserted identity.
\end{proof}

We shall now provide an alternative proof of Theorem \ref{prop:Mecke} based on the explicit distributional description of a spherical splitting tessellation process.

\begin{proof}[Proof based on the distributional description]
	We give a proof which uses only the explicit path distribution described in Section \ref{sec:Prelim}.
	For a deterministic c\`adl\`ag path
	\(\Upsilon=(\Upsilon_u)_{u\in[0,t]}\) of the splitting process, let
	\(m_t(\Upsilon)\) be the corresponding collection of birth-time marked maximal
	polytopes born up to time \(t\). We use the following indexing convention. If
	\(\Upsilon\) has jumps at $0<s_1<\ldots<s_n<t$
	then the jump at time \(s_j\) is caused by a pair $(c_j,S_j)$ with $c_j\in\Upsilon_{s_j-}$, $S_j\in\SS_{d-1}[c_j]$
	and the maximal polytope born at this jump is \(c_j\cap S_j\). Thus $\Upsilon_{s_j}=\oslash_{c_j,S_j,\Upsilon_{s_j-}}$.
	In the notation introduced before the proposition,
	\(m_t(\Upsilon)\wedge(c_j\cap S_j)\) denotes the birth-time marked trace on \(c_j\cap S_j\)
	of those maximal polytopes which are born strictly after time \(s_j\).
	
	By the explicit distributional representation of the process and Tonelli's
	theorem, all terms being non-negative,
	\begin{align}
		L&:=\EE\sum_{(p,s)\in M_t}g(p,s,M_t\wedge p) \notag=\sum_{n=0}^\infty
		\int_{0<s_1<\ldots<s_n<t}
		\int\cdots\int\exp\Big\{-\int_0^t\phi(\Upsilon_u)\,\dint u\Big\}\notag\\
		&\qquad\times
		{\bf 1}\{(\Upsilon_u)_{u\in[0,t]}\in
		\cD(\SS^d;[0,t];(s_\ell,c_\ell,S_\ell)_{1\leq \ell\leq n})\}
		\notag\\
		&\qquad\times
		\sum_{j=1}^n
		g\big(c_j\cap S_j,s_j,m_t(\Upsilon)\wedge(c_j\cap S_j)\big)\,\prod_{\ell=1}^n
		\phi(\Upsilon_{s_\ell-};\dint(c_\ell,S_\ell))\,\dint s_1\ldots\dint s_n.
		\label{eq:MeckeAlternativeStart}
	\end{align}
	
	We now single out the jump which gives rise to the maximal polytope under
	consideration. In the \(j\)-th summand in
	\eqref{eq:MeckeAlternativeStart} put
	\[
	s=s_j,\qquad c=c_j,\qquad S=S_j.
	\]
	The variables with indices \(1,\ldots,j-1\) describe the past before \(s\),
	while those with indices \(j+1,\ldots,n\) describe the future after the imposed
	split of \(c\) by \(S\) at time \(s\). Since the integrations are over ordered
	jump times, summing over all \(j\) is equivalent to first choosing \(s\in(0,t)\),
	then integrating over all possible past histories on \([0,s)\) and all possible
	future histories on \((s,t]\). The factor in the path density belonging to the
	selected jump is $\phi(\Upsilon_{s-};\dint(c,S))$.
	By the definition of the splitting kernel, 
    for every non-negative measurable function
	\[
	F:\{(T,q,S):T\in\TT^d,\ q\in T,\ S\in\SS_{d-1}[q]\}\to\RR,
	\]
	one has
	\begin{equation}\label{eq:MeckeAlternativeKernel}
		\int F(T,q,S)\,\phi(T;\dint(q,S))
		=
		\sum_{q\in T}\int_{\SS_{d-1}[q]}F(T,q,S)\,\kappa(\dint S).
	\end{equation}
	This deterministic identity is precisely the step which turns the selected
	jump into an integral over the cells present at time \(s\) and over the
	hyperspheres hitting the selected cell.
	
	Let us next make explicit what happens to the remaining factors. If the path
	before time \(s\) is fixed and the cell \(c\in\Upsilon_{s-}\) is split by \(S\)
	at time \(s\), we denote the resulting tessellation by $\oslash_{c,S,\Upsilon_{s-}}$.
	Moreover, let \(M^{\oslash_{c,S,\Upsilon_{s-}}}_{s,t}\) be the collection of
	maximal polytopes born in the time interval \((s,t]\) when the process is
	started from \(\oslash_{c,S,\Upsilon_{s-}}\) at time \(s\). The exponential
	factor in \eqref{eq:MeckeAlternativeStart} factorizes as
	\[
	\exp\Big\{-\int_0^t\phi(\Upsilon_u)\,\dint u\Big\}
	=
	\exp\Big\{-\int_0^s\phi(\Upsilon_u)\,\dint u\Big\}
	\exp\Big\{-\int_s^t\phi(\Upsilon_u)\,\dint u\Big\}.
	\]
	The first factor, together with the integrations over the jumps before \(s\),
	gives the law of the process up to time \(s\). The second factor, together with
	the integrations over the jumps after \(s\), is exactly the explicit path
	distribution on the interval \([s,t]\) for the process started from
	\(\oslash_{c,S,\Upsilon_{s-}}\) at time \(s\). Therefore
	\eqref{eq:MeckeAlternativeStart} and \eqref{eq:MeckeAlternativeKernel} yield
	\begin{align}
		L
		&=\EE\int_0^t\sum_{c\in Y_{s-}}
		\int_{\SS_{d-1}[c]}
		\EE\Big[
		g\big(c\cap S,s,
		M^{\oslash_{c,S,Y_{s-}}}_{s,t}\wedge(c\cap S)\big)
		\,\Big|\,\mathfrak{I}_{s-}\Big]\,
		\kappa(\dint S)\,\dint s,
		\label{eq:MeckeAlternativeConditional}
	\end{align}
    where $\mathfrak{I}_{s-}$ denotes the $\sigma$-field generated by $(Y_u)_{0\leq u<s}$.
	Since the jump times have continuous distributions, the set of jump times has
	Lebesgue measure zero almost surely. Hence \(\mathfrak{I}_{s-}\) may be replaced by
	\(\mathfrak{I}_{s}\), the $\sigma$-field generated by $(Y_u)_{0\leq u\leq s}$, inside the integral.
	
	It remains to identify the conditional distribution in
	\eqref{eq:MeckeAlternativeConditional}. Conditional on \(\mathfrak{I}_{s}\) and on the
	additional split of \(c\) by \(S\), the cells of
	\(\oslash_{c,S,Y_s}\) evolve independently after time \(s\). Cells different
	from \(c\cap S^+\) and \(c\cap S^-\) cannot create traces on the new maximal
	polytope \(c\cap S\). Thus only the two daughter cells are relevant for
	\(M^{\oslash_{c,S,Y_s}}_{s,t}\cap(c\cap S)\). By the same explicit path
	distribution, restricted now to each of these two daughter cells and to the time
	interval of length \(t-s\), the evolutions in \(c\cap S^+\) and \(c\cap S^-\)
	have the same laws as the restrictions of two independent copies of the
	original maximal-polytope process during time \(t-s\), with all birth times
	shifted by \(s\). Hence
	\begin{align}
		&M^{\oslash_{c,S,Y_s}}_{s,t}\wedge(c\cap S) \stackrel{d}{=}
		\Big(
		(M_{t-s,+s}^{(1)}\wedge(c\cap S^+))
		\cup
		(M_{t-s,+s}^{(2)}\wedge(c\cap S^-))
		\Big)\wedge(c\cap S),
		\label{eq:MeckeAlternativeFutureLaw}
	\end{align}
	where \((M_u^{(1)})_{u\geq0}\) and \((M_u^{(2)})_{u\geq0}\) are independent
	copies of \((M_u)_{u\geq0}\), independent also of $\mathfrak{I}_{s}$.
	
	Substituting \eqref{eq:MeckeAlternativeFutureLaw} into
	\eqref{eq:MeckeAlternativeConditional} gives
	\begin{align*}
		L
		&=\EE\int_0^t\sum_{c\in Y_s}\int_{\SS_{d-1}[c]}
		g\Big(c\cap S,s,
		\big((M_{t-s,+s}^{(1)}\wedge(c\cap S^+)) \\
		&\hspace{5.2cm}\cup
		(M_{t-s,+s}^{(2)}\wedge(c\cap S^-))\big)\wedge(c\cap S)\Big)\,
		\kappa(\dint S)\,\dint s,
	\end{align*}
	which is the asserted identity.
\end{proof}

\section{Internal incidences on spherical maximal segments for $d=2$}\label{sec:SphericalMaximalSegments}

In this section we consider the number \(\sI_t\) of internal incidences of the typical
maximal edge of a spherical splitting tessellation \(Y_t\) on \(\SS^2\) with
time parameter \(t>0\) and directional distribution \(\kappa\). We restrict
attention to dimension \(d=2\), where an internal incidence is created by the
intersection of two maximal edges and can therefore be described by a single
birth time of the typical maximal edge together with the birth time of the
intersecting segment. In higher dimensions, an internal incidence is 
generated by the intersection of several maximal polytopes, so that one has to
keep track of their joint vector of birth times as well as the corresponding
intersection configuration. This leads to a substantially more involved
structure, as is already apparent for STIT tessellations in \(\RR^d\), see
\cite[Chapter~14.6]{NTWbook}.

We first give a formal definition of \(\sI_t\). For \((e,s)\in M_t\), put
\[
\sI(e,s,M_t) := \big|\{(f,u)\in M_t:u>s,\ f\cap {\rm relint}\ e \neq\varnothing\}\big|. 
\]
Thus \(\sI(e,s,M_t)\) counts the internal incidences on the relative interior of \(e\) generated by maximal edges born after time \(s\). To pass to the count for the typical maximal edge, we use the Palm version of the marked maximal-edge process. Namely, we introduce the random triple \(({\bf e},{\bf s},{\bf M}_t)\) whose distribution is given by
\[
\PP(({\bf e},{\bf s},{\bf M}_t)\in\,\cdot\,) = {1\over N_1(t)} \EE\sum_{(e,s)\in M_t} \int_{\Theta_{z(e)}} {\bf 1}\{(\varrho^{-1}e,s,\varrho^{-1}M_t)\in\,\cdot\,\} \,\nu_{z(e)}(\dint\varrho),
\]
where \(N_1(t)=t^2+1-e^{-t}\) is the expected number of spherical maximal edges of \(Y_t\), see \cite[Equation (7.8)]{HugThaele18}. In words, \(({\bf e},{\bf s},{\bf M}_t)\) is obtained by selecting a maximal edge of \(Y_t\) according to its counting measure, recording its birth time, rotating the selected edge into the reference position, and applying the same rotation to the entire marked configuration \(M_t\). The whole configuration has to be retained, because the number of internal incidences depends on the future maximal edges which hit the selected edge. We then define \[ \sI_t:=\sI({\bf e},{\bf s},{\bf M}_t). \]

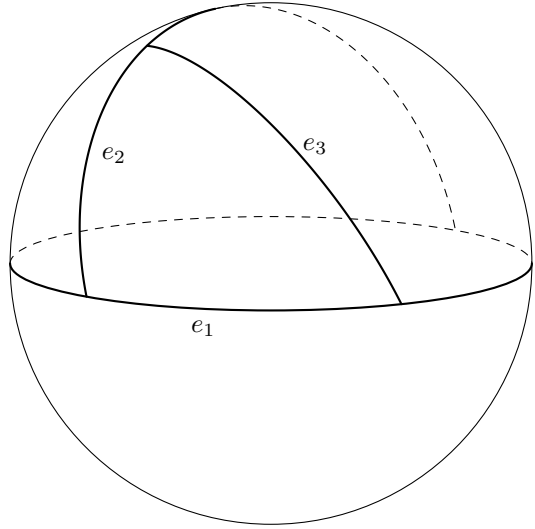
\begin{figure}[t]
\centering
\begin{tikzpicture}[
    scale=1.15,
    line cap=round,
    line join=round,
    every node/.style={font=\small}
]

\def\R{3}
\pgfmathsetmacro{\ct}{0.18}
\pgfmathsetmacro{\st}{sqrt(1-\ct*\ct)}
\pgfmathsetmacro{\eqh}{\R*\ct}

\def\phiTwo{225}
\def\betaTwo{74}
\def\sigmaTwo{-1}

\newcommand{\GreatSemiCircle}[5]{%
    \pgfmathsetmacro{\A}{-\st*sin(#1)}
    \pgfmathsetmacro{\B}{-(#3)*\st*cos(#2)*cos(#1)+\ct*sin(#2)}
    \pgfmathsetmacro{\tzero}{atan2(\A,-\B)}

    \draw[#4]
        plot[domain=0:\tzero,samples=80,variable=\t,smooth]
        ({\R*(cos(#1)*cos(\t) + (#3)*(-sin(#1)*cos(#2))*sin(\t))},
         {\R*(\ct*(sin(#1)*cos(\t) + (#3)*(cos(#1)*cos(#2))*sin(\t))
              + \st*(sin(#2)*sin(\t)))});

    \draw[#5]
        plot[domain=\tzero:180,samples=80,variable=\t,smooth]
        ({\R*(cos(#1)*cos(\t) + (#3)*(-sin(#1)*cos(#2))*sin(\t))},
         {\R*(\ct*(sin(#1)*cos(\t) + (#3)*(cos(#1)*cos(#2))*sin(\t))
              + \st*(sin(#2)*sin(\t)))});
}

\newcommand{\GreatSegmentFromEquatorToEtwo}[3]{%
    \pgfmathsetmacro{\Sx}{cos(#1)}
    \pgfmathsetmacro{\Sy}{sin(#1)}
    \pgfmathsetmacro{\Sz}{0}

    \pgfmathsetmacro{\Qx}{cos(\phiTwo)*cos(#2) + \sigmaTwo*(-sin(\phiTwo)*cos(\betaTwo))*sin(#2)}
    \pgfmathsetmacro{\Qy}{sin(\phiTwo)*cos(#2) + \sigmaTwo*( cos(\phiTwo)*cos(\betaTwo))*sin(#2)}
    \pgfmathsetmacro{\Qz}{sin(\betaTwo)*sin(#2)}

    \pgfmathsetmacro{\dotSQ}{\Sx*\Qx + \Sy*\Qy + \Sz*\Qz}
    \pgfmathsetmacro{\alphaSQ}{acos(\dotSQ)}
    \pgfmathsetmacro{\sinAlphaSQ}{sin(\alphaSQ)}

    \pgfmathsetmacro{\Vx}{(\Qx - \dotSQ*\Sx)/\sinAlphaSQ}
    \pgfmathsetmacro{\Vy}{(\Qy - \dotSQ*\Sy)/\sinAlphaSQ}
    \pgfmathsetmacro{\Vz}{(\Qz - \dotSQ*\Sz)/\sinAlphaSQ}

    \draw[#3]
        plot[domain=0:\alphaSQ,samples=90,variable=\t,smooth]
        ({\R*(\Sx*cos(\t) + \Vx*sin(\t))},
         {\R*(\ct*(\Sy*cos(\t) + \Vy*sin(\t))
              + \st*(\Sz*cos(\t) + \Vz*sin(\t)))});
}

\draw[thin] (0,0) circle[radius=\R];

\draw[dashed,thin]
    (-\R,0) arc[start angle=180,end angle=0,x radius=\R,y radius=\eqh];

\draw[thick]
    (\R,0) arc[start angle=0,end angle=-180,x radius=\R,y radius=\eqh];

\GreatSemiCircle{\phiTwo}{\betaTwo}{\sigmaTwo}{thick}{dashed,thin}

\GreatSegmentFromEquatorToEtwo{300}{65}{thick}

\node[left]  at (-0.52,-0.75) {$e_1$};
\node[left]  at (-1.55,1.25) {$e_2$};
\node[right] at (0.25,1.35) {$e_3$};

\end{tikzpicture}
\caption{Illustration of a splitting tessellation on $\SS^2$ consisting of three maximal edges $e_1$, $e_2$ and $e_3$. The number of internal incidences on $e_1$ is $2$ (one induced by $e_2$ and one induced by $e_3$), that on $e_2$ is $1$ (induced by $e_3$) and that on $e_3$ equals $0$.}
\label{fig:Incidences}
\end{figure}

\subsection{Expectation}

We start by computing the expectation of $\sI_t$. The method we use here will be further developed in the next subsection to determine the exact distribution of $\sI_t$. We assume in this subsection that the directional distribution $\kappa$ is absolutely continuous with respect to the invariant measure.

\begin{theorem}\label{thm:InternalVerticesExpectation}
	Consider the number of internal incidences \(\sI_t\) of the typical maximal
	edge of a spherical splitting tessellation \(Y_t\) on \(\SS^2\)
	with time parameter \(t>0\). Then
	\[
	\EE\sI_t
	=
	{2t^2-2t+2-2e^{-t}\over t^2+1-e^{-t}}.
	\]
\end{theorem}
\begin{proof}
We denote by $N_1(t) = t^2+1-e^{-t}$
the expected number of spherical maximal edges of $Y_t$. Then the definition of Palm distributions yields
\begin{align*}
\EE\sI_t &= {1\over N_1(t)}\EE\sum_{(e,s)\in M_t}\int_{\Theta_{z(e)}}\sI(\varrho^{-1}e,s,\varrho^{-1}M_t)\,\nu_{z(e)}(\dint\varrho)\\
&={1\over N_1(t)}\EE\sum_{(e,s)\in M_t}\sI(e,s,M_t).
\end{align*}
Next, we apply the Mecke-type formula from Theorem \ref{prop:Mecke} and Fubini's theorem to conclude that the last expression is equal to
\begin{align*}
&{1\over N_1(t)}\EE\int_0^t\sum_{c\in Y_s}\int_{\SS_1[c]}\sI(c\cap S,s,(M_{t-s,+s}^{(1)}\wedge(c\cap S^+))\cup(M_{t-s,+s}^{(2)}\wedge(c\cap S^-)))\,\kappa(\dint S)\,\dint s\\
&={1\over N_1(t)}\int_0^t\int_{\SS_1}\EE\sum_{c\in Y_s\atop c\cap S\neq\varnothing}\sI(c\cap S,s,(M_{t-s,+s}^{(1)}\wedge(c\cap S^+))\cup(M_{t-s,+s}^{(2)}\wedge(c\cap S^-)))\,\kappa(\dint S)\, \dint s.
\end{align*}
Now, fix $s\in(0,t)$ and $S\in\SS_1$, and observe that
\begin{align*}
&\sI(c\cap S,s,(M_{t-s,+s}^{(1)}\wedge(c\cap S^+))\cup(M_{t-s,+s}^{(2)}\wedge(c\cap S^-)))\\
&\qquad=\sI(c\cap S,s,M_{t-s,+s}^{(1)}\wedge(c\cap S^+))+\sI(c\cap S,s,M_{t-s,+s}^{(2)}\wedge(c\cap S^-)).
\end{align*}
In the next lemma we determine the distribution of the random variables $\sI(c\cap S,s,M_{t-s,+s}^{(1)}\wedge(c\cap S^+))$ and $\sI(c\cap S,s,M_{t-s,+s}^{(2)}\wedge(c\cap S^-))$. {We denote by $\pi_1:\PP^2\times(0,\infty)\to\PP^2,(p,\beta)\mapsto p$ the projection to the first coordinate.}

\begin{lemma}\label{lem:PoissonDistribution}
Let $c\in\PP^2$ be a spherical polytope, different from $\SS^2$, and $S\in\SS_1[c]$. Then,
$$
\pi_1\big(\big((M_{t-s,+s}^{(1)}\wedge(c\cap S^+))\cup(M_{t-s,+s}^{(2)}\wedge(c\cap S^-))\big)	\wedge (c\cap S)\big)
$$
is the superposition of two independent Poisson point processes on $c\cap S$ with intensity measures $\kappa(\SS_1[\,\cdot\,])(t-s)$. In particular, for all $c\in\PP^2$ the random variables
$$
\sI(c\cap S,s,M_{t-s,+s}^{(1)}\wedge(c\cap S^+))\qquad\text{ and }\qquad\sI(c\cap S,s,M_{t-s,+s}^{(2)}\wedge(c\cap S^-))
$$
are independent and both Poisson distributed with parameter $\kappa(\SS_{1}[c\cap S])(t-s)$.
\end{lemma}
\begin{proof}
To prove the claim we follow the proof of \cite[Theorem 3.6]{DeussHoerrmannThaele} and first notice that the independence and superposition property is a consequence of the fact that the evolution of the spherical splitting tessellation process is independent and identically distributed in disjoint cells. We can thus concentrate on the simple point process {$\pi_1((M_{t-s,+s}^{(1)}\wedge(c\cap S^+))\wedge(c\cap S))$} without loss of generality. To show that this is in fact a Poisson point process with the desired intensity measure, it is enough to show that for any $n\in\NN$ pairwise disjoint spherical segments $B_1,\ldots,B_n$ on $c\cap S$, $c\in\PP^2\setminus\{\SS^2\}$ and $S\in\SS_1[c]$, one has that
\begin{align}\label{eq:PoissonPropertyClaim}
\PP\Big(\bigcap_{i=1}^n\{|{\pi_1((M_{t-s,+s}^{(1)}\wedge B_i)\wedge(c\cap S))}|=0\}\Big) = \prod_{i=1}^n e^{-\kappa(\SS_{1}[B_i])(t-s)}.
\end{align}
In fact, once this has been shown, it follows that 
$$
\PP\Big( \{|{\pi_1((M_{t-s,+s}^{(1)}\wedge B)\wedge(c\cap S))}|=0\}\Big) =   e^{-\kappa(\SS_{1}[B])(t-s)},
$$
whenever $B$ is a finite disjoint union of spherical segments (recall that $\kappa$ is diffuse). An application of \cite[Theorem 2.2]{Kallenberg} then yields the assertion. 

For the proof of \eqref{eq:PoissonPropertyClaim}, we assume that the segments $B_1,\ldots,B_n$ are ordered on $c\cap S$ from one endpoint of $c\cap S$ to the other.
We establish \eqref{eq:PoissonPropertyClaim} by induction over the number $n$. The case $n=1$ is a direct consequence of \cite[Theorem 3.3]{HugThaele18}. Suppose now that the formula holds with $n$ replaced by $n-1$. Since $\kappa$ is absolutely continuous with respect to the invariant measure, we can then apply \cite[Theorem 3.5]{HugThaele18}, which yields
\begin{align*}
&\PP\Big(\bigcap_{i=1}^n\{|{\pi_1((M_{t-s,+s}^{(1)}\wedge B_i)\wedge(c\cap S))}|=0\}\Big) \\
&= e^{-\kappa(\SS_{1}[B])(t-s)} + \sum_{i=1}^{n-1}\kappa([B_i|B_{i+1}])\int_0^{t-s} e^{-u\kappa(\SS_1[B])}\PP(|{\pi_1((M_{t-s-u,+s}^{(1)}\wedge B_1)\wedge(c\cap S))}|=0)\\
&\hspace{8cm}\times\PP(|{\pi_1((M_{t-s-u,+s}^{(1)}\wedge B_2)\wedge(c\cap S))}|=0)\,\dint u,
\end{align*}
where $B$ is the spherical convex hull of $B_1\cup\ldots\cup B_n$, $B^{(1)}:=B_1\cup\ldots\cup B_i$, $B^{(2)}:=B_{i+1}\cup\ldots\cup B_n$ and $[B_i|B_{i+1}]$ is the set of great circles of $\SS^2$ that separate $B_i$ from $B_{i+1}$. By induction hypothesis, we have that
\begin{align*}
&\PP(|{\pi_1((M_{t-s-u,+s}^{(1)}\wedge B_1)\wedge(c\cap S))}|=0)\PP(|{\pi_1((M_{t-s-u,+s}^{(1)}\wedge B_2)\wedge(c\cap S))}|=0)\\
&= \prod_{j=1}^ne^{-\kappa(\SS_{1}[B_j])(t-s-u)}    
\end{align*}
and hence
\begin{align*}
&\PP\Big(\bigcap_{i=1}^n|{\pi_1((M_{t-s,+s}^{(1)}\wedge B_i)\wedge(c\cap S))}|=0\Big) \\
&= e^{-\kappa(\SS_{1}[B])(t-s)} + \Big(\prod_{j=1}^ne^{-\kappa(\SS_{1}[B_j])(t-s)}\Big)\sum_{i=1}^{n-1}\kappa([B_i|B_{i+1}])\int_0^{t-s} e^{-u(\kappa(\SS_1[B])-\sum_{j=1}^n\kappa(\SS_{1}[B_j]))}\,\dint u\\
&= e^{-\kappa(\SS_{1}[B])(t-s)} + \Big(\prod_{j=1}^ne^{-\kappa(\SS_{1}[B_j])(t-s)}\Big)\sum_{i=1}^{n-1}\kappa([B_i|B_{i+1}])\int_0^{t-s} e^{-u\sum_{j=1}^{n-1}\kappa([B_j|B_{j+1}])}\,\dint u,
\end{align*}
where in the last line we used that $\kappa$ is regular.
By regularity, the relevant hitting and separating sets are disjoint up
to \(\kappa\)-null sets, and
\[
\kappa(\SS_1[B])
=
\sum_{j=1}^n\kappa(\SS_1[B_j])
+
\sum_{j=1}^{n-1}\kappa([B_j|B_{j+1}]).
\]
Consequently,
\begin{align*}
	& e^{-\kappa(\SS_1[B])(t-s)}
	+
	\Big(\prod_{j=1}^n
	e^{-\kappa(\SS_1[B_j])(t-s)}\Big)
	\Big(
	1-e^{-\sum_{j=1}^{n-1}
		\kappa([B_j|B_{j+1}])(t-s)}
	\Big)
	\\
	&=
	\Big(\prod_{j=1}^n
	e^{-\kappa(\SS_1[B_j])(t-s)}\Big)
	e^{-\sum_{j=1}^{n-1}
		\kappa([B_j|B_{j+1}])(t-s)}
	\\
	&\qquad\qquad+
	\Big(\prod_{j=1}^n
	e^{-\kappa(\SS_1[B_j])(t-s)}\Big)
	\Big(
	1-e^{-\sum_{j=1}^{n-1}
		\kappa([B_j|B_{j+1}])(t-s)}
	\Big)
	\\
	&=
	\prod_{j=1}^n
	e^{-\kappa(\SS_1[B_j])(t-s)}.
\end{align*}
 This proves \eqref{eq:PoissonPropertyClaim} and completes the proof of the first part of Lemma \ref{lem:PoissonDistribution}.

For $c\in\PP^2\setminus\{\SS^2\}$ the second part immediately follows from the first one, since in this case the number of vertices $M_{t-s,+s}^{(1)}$ induces in the relative interior of $c\cap S$ is the same as the number of incidences. If $c=\SS^2$ one has to observe that $\pi_1((M_{t-s,+s}^{(1)}\wedge(c\cap S^+))\wedge(c\cap S))$ has a pair of antipodal points, which violates the Poisson property. However, this problem disappears when we instead of the number of intersection points on $c\cap S$ consider the number of incidences $\sI(c\cap S,s,\,\cdot\,)$.  This completes the proof.
\end{proof}

\medspace

We can now continue with the proof of
Theorem~\ref{thm:InternalVerticesExpectation}. Define the event $A_s:=\{Y_s=\{\SS^2\}\}$.
Since the lifetime of the initial cell is exponentially distributed with parameter one, $\PP(A_s)=e^{-s}$.
By Lemma~\ref{lem:PoissonDistribution}, conditionally on \(Y_s\) and \(S\),
the expected number of incidences contributed by a cell \(c\) is $2(t-s)\kappa(\SS_1[c\cap S])$.
On \(A_s\), there is only the cell \(\SS^2\), and hence
\[
\sum_{\substack{c\in Y_s\\c\cap S\neq\varnothing}}
\kappa(\SS_1[c\cap S])=1.
\]
Thus the full great circle born at time \(s\) has conditional expected
incidence count \(2(t-s)\).

On the other hand, on the complementary event \(A_s^{\sf c}\), for \(\kappa\)-almost every \(T\in\SS_1\), the two antipodal
points of \(S\cap T\) lie in two distinct cells of \(Y_s\). Consequently,
\[
\sum_{\substack{c\in Y_s\\c\cap S\neq\varnothing}}
\kappa(\SS_1[c\cap S])
=
\int_{\SS_1}
\sum_{\substack{c\in Y_s\\c\cap S\neq\varnothing}}
{\bf 1}\{T\cap c\cap S\neq\varnothing\}\,
\kappa(\dint T)
=
2,
\]
and the total conditional mean is \(4(t-s)\). Using \(\PP(A_s)=e^{-s}\) and
\(N_1(t)=t^2+1-e^{-t}\), we obtain
\begin{align*}
	\EE\sI_t
	&=
	{1\over N_1(t)}
	\int_0^t
	\big(2(t-s)e^{-s}+4(t-s)(1-e^{-s})\big)\,\dint s
	\\
	&=
	{1\over t^2+1-e^{-t}}
	\left(
	2t^2-2\int_0^t(t-s)e^{-s}\,\dint s
	\right)=
	{2t^2-2t+2-2e^{-t}\over t^2+1-e^{-t}},
\end{align*}
which completes the proof.
\end{proof}

It is instructive to observe that $\EE\sI_t<2$ for $t>0$ and
$$
\lim_{t\to\infty}\EE\sI_t = 2,
$$
which is the expected number of internal vertices of the typical maximal segment in a STIT tessellation in $\RR^2$, which in the Euclidean case is the same as the number of internal incidences, see \cite[p.\ 455]{NagelWeissPlane} and \cite{MNW11,Thaele2009}.

\subsection{Full distribution}

For general directional distributions $\kappa$ we were able to determine the expected number $\EE\sI_t$ of internal incidences of the typical maximal edge of a spherical splitting tessellation $Y_t$ on $\SS^2$ with time parameter $t>0$ in the previous section. In this section we determine the full distribution of the random variable $\sI_t$ in the case that $\kappa$ is the uniform distribution on $\SS_{1}$ in terms of the lower incomplete gamma function $\gamma(a,x)$. The proof is again based on the Mecke-type formula in Theorem \ref{prop:Mecke}.

\begin{theorem}\label{thm:ExactDistribution}
	Consider the number of internal incidences \(\sI_t\) of the typical maximal
	edge of an isotropic spherical splitting tessellation \(Y_t\) on \(\SS^2\)
	with time parameter \(t>0\). If 
	\(n\in\{0,1,2,\ldots\}\), then
	\begin{align*}
		\PP(\sI_t=n)
		&=
		{2^n\over n!(t^2+1-e^{-t})}
		\Bigg(
		e^{-t}\gamma(n+1,t)
		+
		2e^{-t}\big(t\gamma(n+1,t)-\gamma(n+2,t)\big)
		\\
		&\hspace{3.5cm}
		+
		2t^2\int_0^1
		u^2{(1-u)^n\over(2-u)^{n+1}}
		\gamma(n+1,t(2-u))\,\dint u
		\Bigg).
	\end{align*}
\end{theorem}
\begin{proof}[Proof of Theorem~\ref{thm:ExactDistribution}]
Put $D_t:=N_1(t)=t^2+1-e^{-t}$, recall the definition of the random triple \(({\bf e},{\bf s},{\bf M}_t)\)
and denote by \(\QQ_t\) its distribution on the space
$$
\YY:=\PP^2_1\times(0,\infty)\times\cF_{\rm fin}(\PP_1^2\times(0,\infty)),
$$
where, as before, $\cF_{\rm fin}(\PP_1^2\times(0,\infty))$ is the space of finite subsets of $\PP_1^2\times(0,\infty)$. Further define the measurable map
$$
F:\YY\to(0,\infty)\times[0,2\pi],(e,u,m)\mapsto(u,\cH^1(e)),
$$
and put \(\WW_t:=\QQ_t\circ F^{-1}\). Thus \(\WW_t\) is the joint distribution
of the birth time and the length of the typical maximal edge of \(Y_t\).
Since the spaces $\YY$ and $(0,\infty)\times[0,2\pi]$ are standard Borel spaces, the disintegration
theorem \cite[Theorem 3.2]{KallenbergFMP} can be applied and yields the existence of a probability kernel
\(K_t((s,r),\,\cdot\,)\) from $(0,t)\times[0,2\pi]$ to $\YY$ satisfying
\begin{align}
\notag\EE h({\bf e},{\bf s},{\bf M}_t) &= \int_\YY h(e,u,m)\,\QQ_t(\dint(e,u,m))\\
&=\int_{(0,t)\times[0,2\pi]}\int_\YY h(e,u,m)\,
K_t((s,r),\dint(e,u,m))\,\WW_t(\dint(s,r)) \label{eq:Disintegration}
\end{align}
for every non-negative measurable test function \(h:\YY\to\RR\). The kernel $K_t$ can be chosen
so that, for \(\WW_t\)-almost every \((s,r)\in(0,t)\times[0,2\pi]\), it is concentrated on triples
\((e,u,m)\in\YY\) with $F(e,u,m)=(s,r)$.

For \(n\in\{0,1,\ldots\}\), we apply \eqref{eq:Disintegration} to the function
\(h(e,u,m)={\bf 1}\{\sI(e,u,m)=n\}\). Since
\(\sI_t=\sI({\bf e},{\bf s},{\bf M}_t)\), this yields
\begin{equation}\label{eq:ItZwischenschritt}
\PP(\sI_t=n)
=
\int_{(0,t)\times[0,2\pi]} q_n(s,r)\,\WW_t(\dint(s,r)),
\end{equation}
where
\[
q_n(s,r)
:=
K_t((s,r),\{(e,u,m)\in\YY:\sI(e,u,m)=n\}).
\]
Equivalently, \(q_n(s,r)\) is a version of the regular conditional probability
\[
q_n(s,r)
=
\PP(\sI_t=n\,|\,({\bf s},\cH^1({\bf e}))=(s,r)),
\]
defined for \(\WW_t\)-almost all \((s,r)\in(0,t)\times[0,2\pi]\), see \cite[Theorem 8.5]{KallenbergFMP}. By the birth-time represen\-tation of the typical maximal edge of $Y_t$ in
\cite[Corollary 7.9]{HugThaele18}, the measure \(\WW_t\) is given by
\[
\WW_t(\dint(s,r))
=
{2s+e^{-s}\over D_t}\,
\PP_{\overline{\sL}_s}(\dint r)\,\dint s,
\qquad 0<s<t,
\]
where $\PP_{\overline{\sL}_s}$ is the length distribution of the typical edge of the isotropic
Poisson great-circle tessellation with intensity \(s\). More precisely, by Theorem \ref{thm:PoissonEdgeLength} we have
\[
(2s+e^{-s})\PP_{\overline{\sL}_s}(\dint r)
=
e^{-s}\delta_{2\pi}(\dint r)
+
2se^{-s}\delta_\pi(\dint r)
+
{2s^2\over\pi}e^{-rs/\pi}{\bf 1}\{0<r<\pi\}\,\dint r .
\]

For a fixed segment of length \(r\in(0,\pi)\), Lemma~\ref{lem:PoissonDistribution}
gives a Poisson incidence count with mean \({2\over\pi}(t-s)r\). Thus
\[
q_n(s,r)
=
{\big({2\over\pi}(t-s)r\big)^n\over n!}
e^{-{2\over\pi}(t-s)r},
\qquad 0<s<t,\quad 0<r<\pi.
\]
For the atom at \(r=\pi\), the same expression gives
\[
q_n(s,\pi)
=
{(2(t-s))^n\over n!}e^{-2(t-s)},\qquad 0<s<t.
\]
For the full great circle, corresponding to the atom at \(r=2\pi\), an
intersecting future maximal edge leaves an antipodal pair of trace points, but
this pair contributes only one incidence. Hence
\[
q_n(s,2\pi)
=
{(2(t-s))^n\over n!}e^{-2(t-s)},\qquad 0<s<t.
\]

Substituting these conditional probabilities and the above expression for
\(\WW_t\) into \eqref{eq:ItZwischenschritt} gives
\begin{align*}
\PP(\sI_t=n)
&=
\int_0^t
{(2(t-s))^n\over n!}e^{-2(t-s)}
{e^{-s}\over D_t}\,\dint s
\\
&\quad+
\int_0^t
{(2(t-s))^n\over n!}e^{-2(t-s)}
{2se^{-s}\over D_t}\,\dint s
\\
&\quad+
\int_0^t\int_0^\pi
{\big({2\over\pi}(t-s)r\big)^n\over n!}
e^{-{2\over\pi}(t-s)r}
{2s^2e^{-rs/\pi}\over\pi D_t}\,
\dint r\,\dint s .
\end{align*}
	
	For the full-circle atom, the change of variables \(u=t-s\) yields
	\begin{align*}
		\int_0^t
		{(2(t-s))^n\over n!}e^{-2(t-s)}
		{e^{-s}\over D_t}\,\dint s
		&=
		{2^ne^{-t}\over n!D_t}
		\int_0^t u^ne^{-u}\,\dint u
		=
		{2^ne^{-t}\over n!D_t}\gamma(n+1,t).
	\end{align*}
	For the atom at \(\pi\), the same substitution gives
	\begin{align*}
		\int_0^t
		{(2(t-s))^n\over n!}e^{-2(t-s)}
		{2se^{-s}\over D_t}\,\dint s
		&=
		{2^{n+1}e^{-t}\over n!D_t}
		\int_0^t u^n(t-u)e^{-u}\,\dint u
		\\
		&=
		{2^{n+1}e^{-t}\over n!D_t}
		\big(t\gamma(n+1,t)-\gamma(n+2,t)\big).
	\end{align*}
	
	It remains to evaluate the continuous part.  Substituting \(r=\pi x\) gives
	\begin{align*}
		&\int_0^t\int_0^\pi
		{\big({2\over\pi}(t-s)r\big)^n\over n!}
		e^{-{2\over\pi}(t-s)r}
		{2s^2e^{-rs/\pi}\over\pi D_t}\,
		\dint r\,\dint s
		=
		{2^{n+1}\over n!D_t}
		\int_0^t\int_0^1
		x^n s^2(t-s)^n e^{-x(2t-s)}
		\,\dint x\,\dint s .
	\end{align*}
	Now put \(s=tu\).  Then
	\begin{align*}
		&{2^{n+1}\over n!D_t}
		\int_0^t\int_0^1
		x^n s^2(t-s)^n e^{-x(2t-s)}
		\,\dint x\,\dint s
		\\
		&=
		{2^{n+1}t^{n+3}\over n!D_t}
		\int_0^1
		u^2(1-u)^n
		\int_0^1 x^n e^{-tx(2-u)}\,\dint x\,\dint u
		\\
		&=
		{2^{n+1}t^2\over n!D_t}
		\int_0^1
		u^2{(1-u)^n\over(2-u)^{n+1}}
		\gamma(n+1,t(2-u))\,\dint u .
	\end{align*}
	
	Adding the three contributions and factoring out
	\(2^n/(n!D_t)\), we obtain
	\begin{align*}
		\PP(\sI_t=n)
		&=
		{2^n\over n!(t^2+1-e^{-t})}
		\Bigg(
		e^{-t}\gamma(n+1,t)
		+
		2e^{-t}\big(t\gamma(n+1,t)-\gamma(n+2,t)\big)
		\\
		&\hspace{3.5cm}
		+
		2t^2\int_0^1
		u^2{(1-u)^n\over(2-u)^{n+1}}
		\gamma(n+1,t(2-u))\,\dint u
		\Bigg),
	\end{align*}
	as claimed.
\end{proof}

\begin{figure}[t]
\centering
\includegraphics[width=0.8\columnwidth]{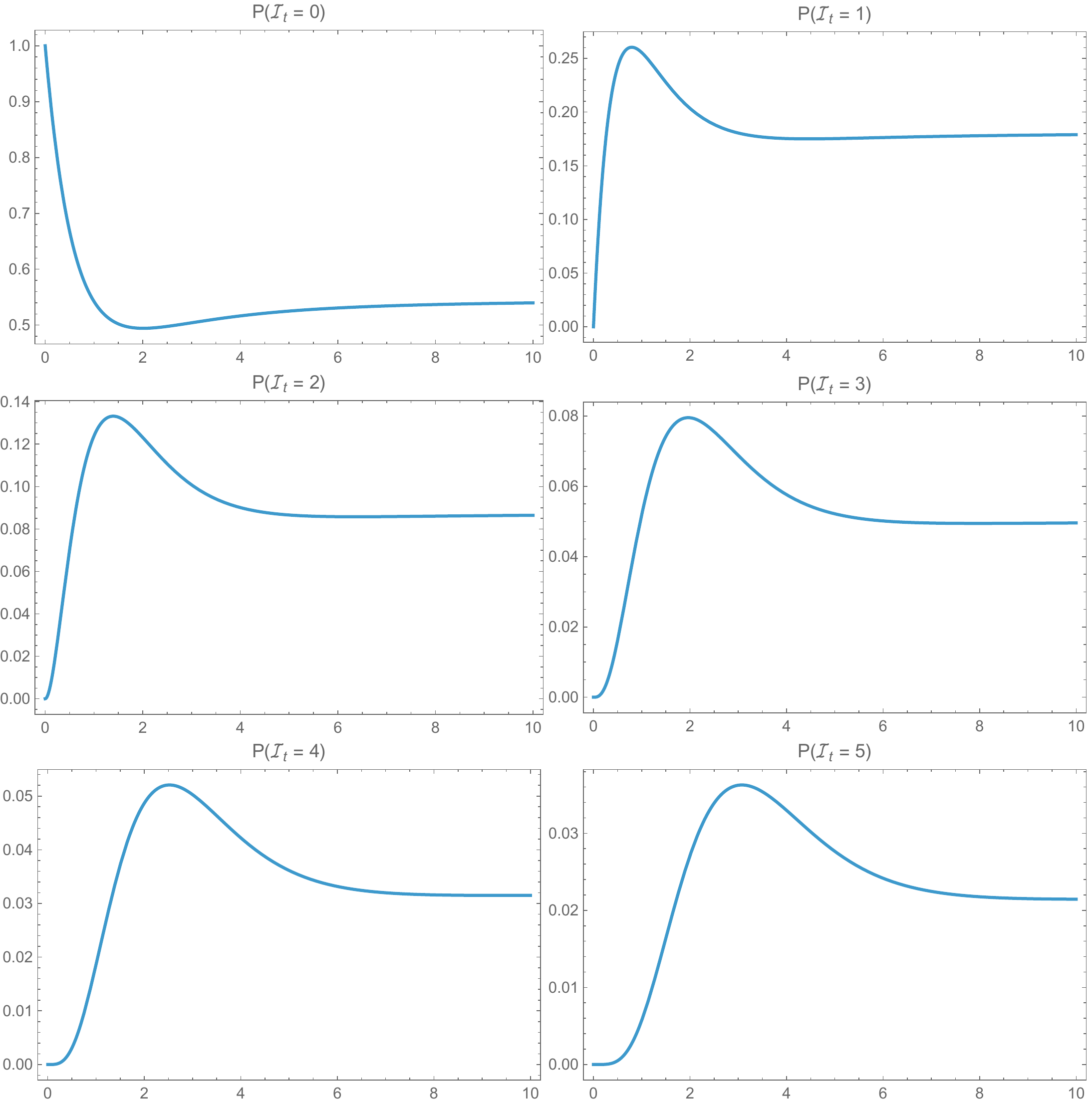}
\caption{The probabilities $\PP(\sI_t=n)$ for $n\in\{0,1,\ldots,5\}$ as a function of $t\in[0,10]$.}
\label{fig:Probabilities}
\end{figure}

For example, using $\gamma(1,x)=1-e^{-x}$,  $\gamma(2,x)=1-(1+x)e^{-x}$,
and evaluating the remaining integrals, we obtain
\begin{align*}
	\PP(\sI_t=0)
	&=
	{1\over t^2+1-e^{-t}}
	\Big(
	8t^2\log 2-5t^2
	-8t^2\big(E_1(t)-E_1(2t)\big)
	\\
	&\hspace{3.5cm}
	+8te^{-t}-4te^{-2t}
	-3e^{-t}+3e^{-2t}
	\Big),
	\\[2mm]
	\PP(\sI_t=1)
	&=
	{2\over t^2+1-e^{-t}}
	\Big(
	16t^2\log 2-11t^2
	-16t^2\big(E_1(t)-E_1(2t)\big)
	\\
	&\hspace{3.5cm}
	+16te^{-t}-7te^{-2t}
	-9e^{-t}+9e^{-2t}
	\Big)
\end{align*}
using the exponential integral function $E_1(t)=\int_t^\infty{e^{-s}\over s}\,\dint s$, see also Figure \ref{fig:Probabilities}. The probabilities $\PP(\sI_t=n)$ with $n\geq 2$ have similar, but increasingly involved expressions. Instead of presenting them, we now determine the limit of $\PP(\sI_t=n)$ as $t\to\infty$.

\begin{corollary}
Consider the number of internal incidences $\sI_t$ of the typical maximal edge of an isotropic spherical splitting tessellation $Y_t$ on $\SS^2$ with time parameter $t>0$. Then, for any $n\in\{0,1,2,\ldots\}$,
$$
\lim_{t\to\infty}\PP(\sI_t=n) = 2^{n+1}\int_0^1u^2{(1-u)^n\over(2-u)^{n+1}}\,\dint u.
$$
\end{corollary}
\begin{proof}
We use the result of Theorem \ref{thm:ExactDistribution} and compute
\begin{align*}
	&\lim_{t\to\infty}
	{2^n\over n!(t^2+1-e^{-t})}
	\Bigg(
	e^{-t}\gamma(n+1,t)
	+
	2e^{-t}\big(t\gamma(n+1,t)-\gamma(n+2,t)\big)
	\\
	&\hspace{3.5cm}
	+
	2t^2\int_0^1
	u^2{(1-u)^n\over(2-u)^{n+1}}
	\gamma(n+1,t(2-u))\,\dint u
	\Bigg)
	\\
	&=
	\lim_{t\to\infty}
	{2^{n+1}t^2\over
		n!(t^2+1-e^{-t})}
	\int_0^1
	u^2{(1-u)^n\over(2-u)^{n+1}}
	\gamma(n+1,t(2-u))\,\dint u
	\\
	&=
	2^{n+1}
	\int_0^1
	u^2{(1-u)^n\over(2-u)^{n+1}}\,\dint u.
\end{align*}
In the last step we applied the dominated convergence theorem to interchange limit and integration and used the fact that $\gamma(n+1,t(2-u))$ converges to $\Gamma(n+1)=n!$ as $t\to\infty$.
\end{proof}

The expression
$$
p(n) := 2^{n+1}\int_0^1u^2{(1-u)^n\over(2-u)^{n+1}}\,\dint u
$$
appearing in the last corollary is the probability that the typical I-segment in a stationary STIT tessellation in $\RR^2$ has exactly $n$ internal vertices, see \cite{MNW11,Thaele2010}. The next result shows that this probability can always be expressed as a rational linear combination of $1$ and $\log 2$ with explicit coefficients.

\begin{lemma}
For any $n\in\{0,1,2,\ldots\}$ one has
$$
p(n) = 2^n(n^2+7n+8)
\left(
\log 2-\sum_{k=1}^n \frac{1}{k2^k}
\right)
-n-5.
$$
\end{lemma}
\begin{proof}
We start by applying the substitution $t=2\,\frac{1-u}{2-u}$.
Then
\[
u=2\,\frac{1-t}{2-t},\qquad
1-u=\frac{t}{2-t},\qquad
2-u=\frac{2}{2-t},\qquad
\dint u=-\frac{2}{(2-t)^2}\,\dint t.
\]
As \(u\) runs from \(0\) to \(1\), \(t\) runs from \(1\) to \(0\). Hence
\[
p(n)
=
8\int_0^1 \frac{t^n(1-t)^2}{(2-t)^3}\,\dint t.
\]
Let $L_n:=n^2+7n+8$
A direct differentiation gives
\[
8\frac{t^n(1-t)^2}{(2-t)^3}
=
L_n\frac{t^n}{2-t}
+
\frac{\dint}{\dint t}
\left[
\frac{t^{n+1}\bigl((n+7)t-2(n+6)\bigr)}{(2-t)^2}
\right].
\]
Therefore
\[
\begin{aligned}
	p(n)
	&=
	L_n\int_0^1 \frac{t^n}{2-t}\,\dint t
	+
	\left[
	\frac{t^{n+1}\bigl((n+7)t-2(n+6)\bigr)}{(2-t)^2}
	\right]_{0}^{1} =
	L_n\int_0^1 \frac{t^n}{2-t}\,\dint t
	-n-5.
\end{aligned}
\]

It remains to evaluate the integral
$$
A_n:=\int_0^1 \frac{t^n}{2-t}\,\dint t.
$$
With \(t=2s\), we obtain
\[
A_n
=
2^n\int_0^{1/2}\frac{s^n}{1-s}\,\dint s.
\]
Since
\[
\frac{s^n}{1-s}
=
\frac{1}{1-s}-\sum_{j=0}^{n-1}s^j,
\]
we get
\[
\begin{aligned}
	A_n
	&=
	2^n\left(
	\int_0^{1/2}\frac{\dint s}{1-s}
	-
	\sum_{j=0}^{n-1}\int_0^{1/2}s^j\,\dint s
	\right) =
	2^n\left(
	\log 2
	-
	\sum_{k=1}^{n}\frac{1}{k2^k}
	\right).
\end{aligned}
\]
This proves the formula.
\end{proof}

\begin{remark}
Since $\log 2 = \sum_{k=1}^\infty{1\over k2^k}$, the formula for $p(n)$ can alternatively be written as
$$
p(n) = 2^n(n^2+7n+8)\sum_{k=n+1}^\infty{1\over k2^k} - n - 5.
$$
\end{remark}

\subsection*{Acknowledgement}
CT and DH have been supported by the DFG priority program SPP 2265 \textit{Random Geometric Systems}.

During the preparation of this work, the authors used ChatGPT (OpenAI) as tools for technical and editorial assistance. All mathematical statements and proofs were independently verified and finalized by the authors.

\addcontentsline{toc}{section}{References}


\end{document}